\documentclass[final,a4paper]{elsarticle}

\usepackage{amsmath}
\usepackage{amssymb,amsthm,amsfonts}
\usepackage{latexsym}
\usepackage{euscript}

\usepackage{mathrsfs}
\usepackage{mathptmx}
\usepackage{setspace}
\usepackage{paralist}
\usepackage{float}
\usepackage{titlesec}
\usepackage{graphicx}
\usepackage{pgf,tikz}
\usetikzlibrary{arrows}
\tikzset{every mark/.append style={scale=0.1}}
\usepackage{hyperref}

\newtheoremstyle{neu_thm}% name
{13pt}       % Space above
{8pt}      % Space below
{\itshape}  % Body font
{}          % Indent amount (empty = no indent, \parindent = para indent)
{\bfseries} % Thm head font
{.}         % Punctuation after thm head
{.5em}      % Space after thm head: " " = normal interword space;
{}          % Thm head spec (can be left empty, meaning 'normal')

\newtheoremstyle{neu_defn}% name
{13pt}       % Space above
{8pt}      % Space below
{}  % Body font
{}          % Indent amount (empty = no indent, \parindent = para indent)
{\bfseries} % Thm head font
{.}         % Punctuation after thm head
{.5em}      % Space after thm head: " " = normal interword space;
{}          % Thm head spec (can be left empty, meaning 'normal')

\theoremstyle{neu_thm}
\newtheorem{thm}{Theorem}[section]
\newtheorem{cor}[thm]{Corollary}
\newtheorem{lem}[thm]{Lemma}
\newtheorem{prop}[thm]{Proposition}

\theoremstyle{neu_defn}

\newtheorem{rem}[thm]{Remark}
\newtheorem{ex}[thm]{Example}

\numberwithin{equation}{section}

\newcounter{counter_a}
\newenvironment{myenum}{\begin{list}{\textup{(\roman{counter_a})}}%
{\usecounter{counter_a}
\setlength{\itemsep}{0.5ex}\setlength{\topsep}{0.7ex}
\setlength{\leftmargin}{5ex}\setlength{\labelwidth}{5ex}}}{\end{list}}

\newcounter{counter_b}
\newenvironment{myenuma}{\begin{list}{{\rm(\alph{counter_b})}}%
{\usecounter{counter_b}
\setlength{\itemsep}{0.5ex}\setlength{\topsep}{0.7ex}
\setlength{\leftmargin}{5ex}\setlength{\labelwidth}{5ex}}}{\end{list}}

\newcounter{counter_d}
\newenvironment{myenumarabic}{\begin{list}{\textup{(\arabic{counter_d})}}%
{\usecounter{counter_d}
\setlength{\itemsep}{0.5ex}\setlength{\topsep}{0.7ex}
\setlength{\leftmargin}{5ex}\setlength{\labelwidth}{5ex}}}{\end{list}}

\titleformat{\section}{\normalfont\bfseries\centering}{\thesection.}{.25em}{}
\titleformat{\subsection}{\normalfont\bfseries}{\thesubsection.}{.25em}{}
\allowdisplaybreaks
\makeatletter
\def\moverlay{\mathpalette\mov@rlay}
\def\mov@rlay#1#2{\leavevmode\vtop{%
   \baselineskip\z@skip \lineskiplimit-\maxdimen
   \ialign{\hfil$\m@th#1##$\hfil\cr#2\crcr}}}
\newcommand{\charfusion}[3][\mathord]{
    #1{\ifx#1\mathop\vphantom{#2}\fi
        \mathpalette\mov@rlay{#2\cr#3}
      }
    \ifx#1\mathop\expandafter\displaylimits\fi}
\makeatother

\DeclareMathOperator{\dist}{dist}

\newcommand{\sigmap}{\sigma_{\textup{p}}}
\newcommand{\sigmaess}{\sigma_{\textup{ess}}}
\newcommand{\rd}{\mathrm{d}}
\newcommand{\eps}{\varepsilon}

\newcommand{\defeq}{\mathrel{\mathop:}=}

      \def\dC{{\mathbb C}}

   \def\dN{{\mathbb N}}   
      \def\dR{{\mathbb R}}

      \def\cF{{\mathcal F}}
   \def\cH{{\mathcal H}}   
   \def\cK{{\mathcal K}}   
\def\cM{{\mathcal M}}      
      
\def\cS{{\mathcal S}}   \def\cT{{\mathcal T}}

\DeclareMathOperator\diag{diag}

 \def\ker{{\xker\,}}
\newcommand{\Pluspunkt}{\stackrel{.}{+}}

\newcommand{\Skdef}{(\raisebox{0.5 ex}{.},\raisebox{0.5 ex}{.})}
\newcommand{\Skindef}{[\raisebox{0.5 ex}{.},\raisebox{0.5 ex}{.}]}

\newcommand{\mymatrix}[1]{\left[\;\begin{matrix} #1 \end{matrix}\;\right]}

\newcommand{\ra}{\rightarrow}

\newcommand{\K}[2]{\left[ #1, #2\right]}
\newcommand{\ort}{[\perp]}

\DeclareMathOperator{\linspan}{span}

\newcommand\void[1]{}

\definecolor{darkgreen}{rgb}{0.,0.6,0.}
\definecolor{enclcol}{rgb}{1.,0.5,0.2}     % orange
\definecolor{mygreen}{rgb}{0.,0.6,0.}      % dark green
\definecolor{myorange}{rgb}{1.,0.6,0.2}    % orange
\definecolor{mygrey}{rgb}{0.26,0.26,0.26}  % dark grey
\definecolor{myblue}{rgb}{0.,0.6,1.}       % light blue
\definecolor{myred}{rgb}{0.86,0.078,0.23}  % red

{\left[\begin{smallmatrix}}%            begin code
{\end{smallmatrix}\right]}%

\newcommand{\ol}{\overline}
\newcommand{\ds}{\dotplus}
\newcommand{\wt}{\widetilde}

\renewcommand{\emptyset}{\varnothing}
\newcommand{\rmref}[1]{{\rm\ref{#1}}}

\newcommand{\braces}[1]{{\rm (}#1{\rm )}}

\newcommand{\R}{\ensuremath{\mathbb R}}    % Reelle Zahlen
\newcommand{\C}{\ensuremath{\mathbb C}}    % Komplexe Zahlen
\newcommand{\N}{\ensuremath{\mathbb N}}    % Nat"urliche Zahlen
\newcommand{\gperp}{{[\perp]}}

\newcommand{\product}{[\cdot\,,\cdot]}

\newcommand{\calB}{\mathcal B}

\newcommand{\calH}{\mathcal H}

\newcommand{\calK}{\mathcal K}         
\newcommand{\calL}{\mathcal L}         
\newcommand{\calM}{\mathcal M}

\newcommand{\la}{\lambda}

\newcommand{\bmat}[4]
{
   \mymatrix{#1 & #2\\#3 & #4}
}

\newcommand{\smallmat}[4]{\left(\begin{smallmatrix}#1 & #2\\#3 & #4\end{smallmatrix}\right)}

\renewcommand{\Im}{\operatorname{Im}\,}
\renewcommand{\Re}{\operatorname{Re}\,}
\newcommand{\RE}{\Re}

\renewcommand{\ker}{\operatorname{ker}\,}
\newcommand{\ran}{\operatorname{ran}\,}
\newcommand{\dom}{\operatorname{dom}\,}
\newcommand{\codim}{\operatorname{codim}\,}
\newcommand{\rank}{\operatorname{rank}\,}

\newcommand{\Sra}{\Rightarrow}

\definecolor{darkgreen}{rgb}{0,0.6,0.1}

\newcommand{\red}[1]{{\color{red}{#1}}}

\begin{document}

\begin{frontmatter}

\title{Spectral enclosures for a class of block operator matrices}

\author[IAM,FI]{Juan Giribet}\ead{jgiribet@fi.uba.ar}
\author[US]{Matthias Langer}\ead{m.langer@strath.ac.uk}
\author[IAM,LP]{Francisco Mart\'{\i}nez Per\'{\i}a \corref{fmp}}\ead{francisco@mate.unlp.edu.ar}

\author[KU]{Friedrich Philipp}\ead{fmphilipp@gmail.com}
\author[IAM,TU]{Carsten Trunk}\ead{carsten.trunk@tu-ilmenau.de}

\address[IAM]{Instituto Argentino de Matem\'{a}tica ``Alberto P. Calder\'{o}n'' (CONICET), Saavedra 15 (1083) Buenos Aires, Argentina}
\address[FI]{Departamento de Matem\'atica -- FI-Universidad de Buenos Aires, Paseo Col\'on 850 (1063) Buenos Aires, Argentina}

\address[US]{Department of Mathematics and Statistics, University of Strathclyde, 26 Richmond Street, Glasgow G1 1XH, United Kingdom \\ \url{personal.strath.ac.uk/m.langer}}

\address[LP]{
Centro de Matem\'atica de La Plata -- FCE, Universidad Nacional de La Plata, C.C.\ 172, (1900) La Plata, Argentina}

\address[KU]{Katholische Universit\"{a}t Eichst\"{a}tt-Ingolstadt,
 Ostenstra\ss e 26,
85072 Eich\-st\"{a}tt, Germany \\ \url{www.ku.de/?fmphilipp}}

\address[TU]{Institut f\"ur  Mathematik, Technische Universit\"{a}t Ilmenau, Postfach 100565, D-98684 Ilmenau, Germany \\ \url{www.tu-ilmenau.de/de/analysis/team/carsten-trunk}}

\cortext[fmp]{Corresponding author}

%\subjclass[2010]{Primary 47A10; Secondary 47B50, 47A12, 81Q15}
%\thanks{
%}

%\date{\today}

%\dedicatory{}

\begin{abstract}
We prove new spectral enclosures for the non-real spectrum of
a class of $2\times2$ block operator matrices
with self-adjoint operators $A$ and $D$ on the diagonal
and operators $B$ and $-B^*$ as off-diagonal entries.
One of our main results resembles Gershgorin's circle theorem.
The enclosures are applied to $J$-frame operators.
\end{abstract}

\end{frontmatter}

%\maketitle

% *******************************************************************
\section{Introduction}
% *******************************************************************

\noindent
We consider block operator matrices $S$ acting in
the orthogonal sum $\calH \defeq \calH_+ \oplus \calH_-$ of two Hilbert spaces,
\begin{equation} \label{PlazaMalvinas}
  S = \mymatrix{ A & B \\ -B^* & D },
\end{equation}
where $A$ and $D$ are (possibly unbounded) self-adjoint operators in $\calH_+$
and $\calH_-$, respectively, and $B$ is a bounded operator from $\calH_-$ to $\calH_+$.

Such operators play an important role in various applications. For instance,
%such block operator matrices
they appear in the study of so-called floating singularities
\cite{Bog85,JT02,JT07,LLT02,LMM90}, in the perturbation theory for equations
of indefinite Sturm--Liouville type \cite{BPT13}, and also
in frame theory \cite{GLLMMT,GMMM12}.

Clearly, $S$ is not self-adjoint in $\calH$ unless $B=0$.
However, it is self-adjoint if we introduce the indefinite inner product
\begin{equation}\label{indefIP}
  \left[\mymatrix{ x_+ \\ x_- }, \mymatrix{ y_+ \\ y_- }\right]
  = (x_+,y_+) - (x_-,y_-), \qquad
  \mymatrix{ x_+ \\ x_- }, \mymatrix{ y_+ \\ y_- } \in \calH;
\end{equation}
for bounded $S$ this means that $[Sx,y]=[x,Sy]$ for all $x,y\in\calH$.
The indefinite inner product $[\cdot,\cdot]$ turns $\calH$ into a Krein space,
i.e.\ it is the orthogonal sum of a Hilbert space and an anti-Hilbert space.
Actually, every bounded self-adjoint operator in a Krein space can be written in
the form \eqref{PlazaMalvinas}.

Of particular interest is the location of the spectrum of $S$. In \cite{LLMT05,LS17,T09}
spectral enclosures were obtained via the quadratic numerical range,
and in \cite{AMT10,BPT13} in terms of the spectra of $A$ and $D$.
Gershgorin-type results for more general operator matrices were presented in \cite{DMS15}
and \cite{Salas99}.
Moreover, in \cite{AL95,JT02,LMM97,LMM90} the essential spectrum was investigated,
and in \cite{LLT02} variational principles and estimates for eigenvalues were proved.
Invariant subspaces and factorizations of Schur complements were considered
in \cite{MS96} and \cite{AM16},
and in \cite{AMS09} conditions were presented for an operator of the form \eqref{PlazaMalvinas}
to be similar to a self-adjoint operator in a Hilbert space.
For an overview we refer to the monograph~\cite{Tretter08}.

%In general, the operator $S$ is not self-adjoint with respect to the standard Hilbert space
%inner product $\Skdef$ in $\calH_+ \oplus \calH_-$, but it is self-adjoint with respect to the indefinite inner product induced by the block operator matrix
%\[
 % J=\matriz{I}{0}{0}{-I}.
%\]
%In other words, if we introduce the indefinite inner product
%\begin{equation}\label{Saveedra}
 % [x, y] \defeq (Jx, y), \qquad x,y \in \calH,
%\end{equation}
%the space $\calH$ turns into a Krein space and
%the operator $S$ is self-adjoint in the Krein space $(\calH,\Skindef)$;
%see, e.g.\ \cite{AI89,B}.
%Conversely, every bounded self-adjoint operator in a Krein space
%can be written in the form \eqref{PlazaMalvinas} with respect to
%a fundamental decomposition of the Krein space.
%Let us point out that self-adjoint operators in Krein spaces
%have (in general) non-real spectrum.

In general, the spectrum of block operator matrices as in \eqref{PlazaMalvinas}
is not contained in the real line.  The self-adjointness of $S$ in the Krein space
with the inner product \eqref{indefIP} implies only that the spectrum of $S$
is symmetric with respect to the real axis.
The aim of this paper is to prove enclosures for the (non-real) spectrum of $S$
in terms of (spectral) quantities of the operators $A$, $B$, and $D$.
%, such as the minimum or maximum of the spectrum and the size of certain gaps in the spectrum.

%We start with a general enclosure for the spectrum formulated in terms of the numerical ranges
%of $A$ and $D$ and the norm of $B$; see Theorem~\ref{thm:spec_incl} below.
%The proof uses the quadratic numerical range, which was introduced in \cite{LT98};
%this leads also to estimates of the norm of the resolvent.
%Theorem~\ref{thm:spec_incl} is sharp in the sense that the enclosures
%cannot be improved if just the numerical ranges of $A$ and $D$ and the norm
%of $B$ are known; see Proposition~\ref{pr:encl_sharp}.
We start with a general enclosure for the (closure) of the quadratic numerical range of $S$, formulated in terms of the numerical ranges
of $A$ and $D$ and the norm of $B$; see Proposition~\ref{pr:spec_incl} below. The quadratic numerical range of a block operator matrix was introduced in \cite{LT98} and its closure contains the spectrum of $S$; see \eqref{W2specincl}. Although similar enclosures for the spectrum of $S$ were already known, one of the advantages of having a spectral enclosure for the quadratic numerical range is that it leads also to estimates of the norm of the resolvent; see the discussion in Remark \ref{re:encl_lit}.
Moreover, Proposition~\ref{pr:spec_incl} is sharp in the sense that the enclosures for the quadratic numerical range
cannot be improved if just the numerical ranges of $A$ and $D$ and the norm
of $B$ are known; see Theorem~\ref{th:encl_sharp}.

The main contribution of this paper is a spectral enclosure for the operator matrix $S$,
which is connected with the Schur complements. It is well known and follows
from a relatively simple Neumann series type argument applied to the first and second Schur complement  that
\begin{equation}\label{toalla}
  \sigma(S)\setminus\R\,\subseteq\,\big\{\la\in\C\setminus\R :
  \|B^*(A-\la)^{-1}B(D-\la)^{-1}\| \ge 1 \mbox{ and }
  \|B(D-\la)^{-1}B^*(A-\la)^{-1}\|
  \ge 1\big\};
\end{equation}
see \cite[Theorem~1.1]{DMS15}, \cite[Lemma~5.2\,(ii)]{AMT10}
or \cite[Section 2.3]{Tretter08}.
Here we prove that
\begin{equation}\label{encl_intro}
  \sigma(S)\setminus\R\,\subseteq\,\big\{\la\in\C\setminus\R : \|(A-\la)^{-1}B\|\ge 1\quad\text{and}\quad\|(D-\la)^{-1}B^*\|\ge 1\big\};
\end{equation}
see Theorem \ref{t:new} below. The enclosures \eqref{toalla} and \eqref{encl_intro}
are independent of each other in the sense that none of the sets in the right-hand sides of \eqref{toalla} and \eqref{encl_intro} is strictly contained in the other one. However, since the norm inequalities in \eqref{encl_intro} deal separately with the resolvent functions of $A$ and $D$ it is easy to construct examples where the enclosure in \eqref{encl_intro} is strictly contained in the one in \eqref{toalla}, see Example \ref{ex:gershgorin}.
%in Section  \ref{smalll} we provide an example
%where the enclosure in \eqref{encl_intro} is better then the one in  \eqref{toalla}.

Note that the spectral enclosures in \eqref{toalla} and \eqref{encl_intro} are not explicitly formulated in terms of the spectra of $A$ and $D$.
 However, it is one of our main observations that \eqref{encl_intro}
 allows a reformulation in a more geometric manner. In particular,
%we show in Section  \ref{smalll} that
given $\la\in\dC\setminus\dR$, $\|(A-\la)^{-1}B\|\ge 1$ if and only if
 for all positive continuous functions $f:\dR\ra\dR$
with some specific behaviour at infinity
we have
\[
  \la\in\bigcup_{t\in\sigma_B(A)}\mathcal{B}_{f(t)^{-1}\|f(A)B\|}(t),
\]
where $\sigma_B(A)$ is a specific closed subset of $\sigma(A)$
and $\mathcal{B}_r(t)$ stands for the closed ball of radius $r$ around $t$;
 for details we refer to Proposition \ref{p:late} below. A similar interpretation can be obtained for the inequality $\|(D-\la)^{-1}B^*\|\ge 1$ in terms of continuous functions defined in a closed subset $\sigma_{B^*}(D)$ of the spectrum of $D$.
 Therefore, the enclosure for the non-real part of the spectrum of $S$ in \eqref{encl_intro} implies a family of enclosures which resemble Gershgorin's circle theorem:
%allows a reformulation which resembles Gershgorin's circle theorem (see Theorem \ref{t:gersh_f2} below):
for any two positive continuous functions $f$ and $g$
with some specific behaviour at infinity we have that
%the following spectral enclosure for the non-real spectrum of $S$ holds:
\begin{equation}\label{Pappnasenbaer}
  \sigma(S)\setminus\R\,\subseteq\,
  \left(\bigcup_{t\in\sigma_B(A)}\mathcal{B}_{f(t)^{-1}\|f(A)B\|}(t)\right)\,\cap\,
  \left(\bigcup_{s\in\sigma_{B^*}(D)}\mathcal{B}_{g(s)^{-1}\|g(D)B^*\|}(s)\right);
\end{equation}
see Theorem \ref{t:gersh_f2} below.

Maybe the most interesting situations appear when one chooses the functions $f$ and $g$ explicitly. For instance, if
%A particular easy situation is when
$A$ is boundedly invertible and $f(t) = |t|^{-1}$ then \eqref{Pappnasenbaer} implies
\begin{equation}\label{trasnoche}
  \sigma(S)\setminus\dR \subseteq
  \bigcup_{a\in\sigma_B(A)} \mathcal{B}_{|a|\|A^{-1}B\|}(a).
\end{equation}
Moreover, if $\|A^{-1}B\|<1$ then these balls are contained in a double-sector with half opening angle $\arcsin\|A^{-1}B\|$, see Figure \ref{fig:gershgorin17} below.
 %defined by a pair of lines
%through the origin which are tangent to the balls in the union.
%
%we refer also to the
\begin{figure}[ht]
\begin{center}
\begin{tikzpicture}[scale=1.0]
% \sigma(A) = [-4,-3.5] \cup \{-2.2\} \cup \{-0.8\} \cup [1,2.5]
% \|A^{-1}B\| = 0.5, half opening angle: \pi/6
  \draw[black] (-5,-1.291) -- (5,1.291);
  \draw[black] (-5,1.291) -- (5,-1.291);
 % \draw[black] (-0.6,-1.5) -- (-0.6,1.5);
 % \draw[black] (0.75,-1.5) -- (0.75,1.5);
  \filldraw[color=enclcol, thick, fill=enclcol!10] (0.9375,0.242) arc (104.5:255.5:0.25)
    -- (2.344,-0.605) arc (-104.5:104.5:0.625) -- cycle;
  \draw[enclcol!30, thin] (1,0) circle (0.25);
  \draw[enclcol!30, thin] (2.5,0) circle (0.625);
  \filldraw[color=enclcol, thick, fill=enclcol!10] (-0.8,0) circle (0.2);
  \filldraw[color=enclcol, thick, fill=enclcol!10] (-3.75,0.968)
    arc (75.5:284.5:1) -- (-3.28,-0.847) arc (-75.5:-18.8:0.875)
    arc (-149.1:149.1:0.55) arc (18.8:75.5:0.875) -- cycle;
  \draw[enclcol!30, thin] (-3.5,0) circle (0.875);
  \draw[enclcol!30, thin] (-4,0) circle (1);
  \draw[enclcol!30, thin] (-2.2,0) circle (0.55);
  \draw[->] (-6.5,0) -- (5.5,0);
  \draw[->] (0,-2) -- (0,2);
  \draw[color=blue, line width=1mm] (1,0) -- (2.5,0);
  \draw[color=blue, line width=1mm] (-4,0) -- (-3.5,0);
  \filldraw[blue] (-2.2,0) circle (0.05);
  \filldraw[blue] (-0.8,0) circle (0.05);
  %\node[color=blue] at (-3.75,-0.3) {$\sigma(A)$};
  %\node[color=blue] at (2.0,-0.3) {$\sigma(A)$};
 % \draw[->, color=blue] (-1.3,-1) -- (-0.83,-0.1);
 % \draw[->, color=blue] (-1.5,-1) -- (-2.15,-0.1);
 % \node[color=blue] at (-1.35,-1.25) {$\sigma(A)$};
\end{tikzpicture}
\caption{The spectral enclosure for $\sigma(S)\setminus\dR$ given in \eqref{trasnoche}
(in orange) and the set $\sigma_B(A)$ (in blue).}
\label{fig:gershgorin17}
\end{center}
\end{figure}
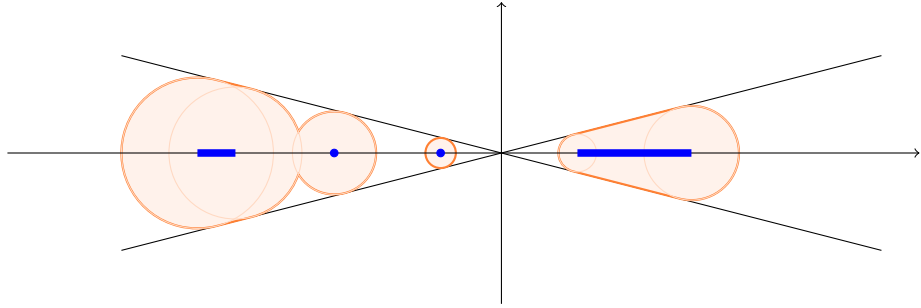

It is worth mentioning that \eqref{encl_intro} also improves the following spectral enclosure obtained in \cite{BPT13}: if $\mathcal{B}_r(\Delta) = \{z\in\C : \dist(z,\Delta)\le r\}$ then
\[
  \sigma(S)\setminus\R\,\subseteq\,\mathcal{B}_{\|B\|}(\sigma(A))\,\cap\,\mathcal{B}_{\|B\|}(\sigma(D)).
\]
In fact, $\|(A-\la)^{-1}B\|\ge 1$ obviously implies $\|(A-\la)^{-1}\|\|B\|\ge 1$ and the latter is equivalent to
$\la\in \mathcal{B}_{\|B\|}(\sigma(A))$. A similar argument with $D$ and $B^*$ instead of $A$ and $B$ completes the proof.

%Finally, in Section~\ref{sec:jframe} we apply the spectral enclosures obtained in Section \ref{smalll} to a class of operators, which can be represented as in \eqref{PlazaMalvinas}, that appears in frame theory.
%The so-called $J$-frame operators were introduced in \cite{GMMM12}. In this case, our findings lead to significant improvements
%of the spectral enclosures for $J$-frame operators obtained in \cite{GLLMMT}.

Finally, in Section~\ref{sec:jframe} we apply the spectral enclosures obtained in Section~\ref{smalll} to operator matrices of the form \eqref{PlazaMalvinas} which appear in frame theory. The so-called $J$-frame operators were introduced in \cite{GMMM12} and further investigated in \cite{GLLMMT}. Our findings lead to significant improvements of the spectral enclosures for $J$-frame operators obtained in \cite{GLLMMT}.

% *******************************************************************
\section{Preliminaries}\label{prelim}
% *******************************************************************

\noindent
If $\calH$ and $\calK$ are Hilbert spaces, we denote by $L(\calH,\calK)$
the space of all bounded linear operators mapping from $\calH$ to $\calK$.
As usual, we set $L(\calH) \defeq L(\calH,\calH)$.
For $r\ge 0$ and $\Delta\subseteq\C$ we set
\[
  \mathcal{B}_r(\Delta) \defeq \{z\in\C : \dist(z,\Delta)\le r\}.
\]
If $a\in\C$, we also write $\mathcal{B}_r(a) \defeq \mathcal{B}_r(\{a\})$ for the closed disc
with centre $a$ and radius $r$.

The numerical range of a linear operator $T$ in the Hilbert space $\calH$ is defined by
\[
  W(T) \defeq \{(Tx,x) : x\in\dom T,\,\|x\|=1\}.
\]
It is well known that the numerical range $W(T)$ is convex and that $\C\setminus\ol{W(T)}$ has
at most two (open) connected components (see \cite[V.3.2]{K95}).
Moreover, it is immediate from the definition of $W(T)$ that $\sigmap(T)\subseteq W(T)$,
where $\sigmap(T)$ stands for the point spectrum of $T$.
If $T$ is closed, $\la\in\C\setminus\ol{W(T)}$
and $x\in\dom T$, $\|x\|=1$, then
\[
  \|(T-\la)x\|\,\ge\,|(Tx,x)-\la|\,\ge\,\dist(\la,\ol{W(T)}).
\]
This shows that $\ran(T-\la)$ is closed and $\ker(T-\la) = \{0\}$.
Hence, if each of the (at most two) components of $\C\setminus\ol{W(T)}$ contains points
from the resolvent set $\rho(T)$, then $\sigma(T)\subseteq\ol{W(T)}$.
This holds in particular if $T$ is bounded.  The next lemma is now immediate.

\begin{lem}\label{l:kato}
Let $T = A+B$, where $A$ is a self-adjoint operator in a Hilbert space $\cH$ and $B\in L(\cH)$.
Then $\sigma(T)\subseteq\ol{W(T)}$ and for $\la\notin\ol{W(T)}$ we have
\begin{equation*}%\label{e:WT}
  \|(T-\la)^{-1}\|\,\le\,\frac 1{\dist(\la,\ol{W(T)})}\,.
\end{equation*}
\end{lem}
%\begin{proof}
%Let $\la\in\C\setminus\ol{W(T)}$, $c\defeq\dist(\la,\ol{W(T)}) > 0$. Then for any $x\in\dom T = \dom A$, $\|x\|=1$, we have $c\le|(Tx,x) - \la| = |(T-\la)x,x)|\le\|(T-\la)x\|$. Hence, $T-\la$ is bounded below and \eqref{e:WT} follows if only $\la\in\rho(T)$. The latter is a consequence of \cite[Theorem V.3.2]{K95} and the fact that $\la\in\rho(T)$ for $|\Im\la| > \|B\|$.
%\end{proof}

%It is a convex subset of $\dC$ %(by T\"{o}plitz-Haussdorf theorem)
%and it satisfies the so-called spectral inclusion property:
%\[
%\sigma_p(A)\subseteq W(A), \quad \sigma(A)\subseteq \ol{W(A)},
%\]
%for the point spectrum $\sigma_p(A)$ and the spectrum $\sigma(A)$, respectively.

Let us also recall the definition of the quadratic numerical range,
which was introduced in \cite{LT98}; see also \cite[Definition~1.1.1]{Tretter08}.
Assume that the Hilbert space $\cH$ is the orthogonal sum of two Hilbert spaces, $\cH_+$ and $\cH_-$.
Let $S$ be a bounded operator in $\cH$ decomposed as
\[
  S= \mymatrix{A & B \\
  C & D},
\]
where $A \in L(\cH_+)$, $B \in L(\cH_-,\cH_+)$, $C \in L(\cH_+,\cH_-)$, and $D \in L(\cH_-)$.
For $f\in\cH_+$ and $g\in\cH_-$ with $\|f\|=\|g\|=1$
we introduce the $2 \times 2$ matrix
\begin{equation}\label{defSfg}
  S_{f,g} = \mymatrix{
    (Af,f) & (Bg,f) \\[1ex]
    (Cf,g) & (Dg,g) }.
\end{equation}
The set 	
\[
  W^2(S) \defeq \bigcup_{\substack{f\in\cH_+,\,g\in\cH_- \\[0.3ex] \|f\|=\|g\|= 1}} \sigmap(S_{f,g})
\]
is called the \emph{quadratic numerical range of} $S$.
It is no longer a convex subset of $\dC$, but it has at most two connected components.

One of the advantages of the quadratic numerical range is that it is contained in the numerical range: $W^2(S)\subseteq W(S)$ and that we have the following refined spectral inclusions
\begin{equation}\label{W2specincl}
  \sigmap(S) \subseteq W^2(S)
  \qquad\text{and}\qquad
  \sigma(S) \subseteq \ol{W^2(S)};
\end{equation}
see \cite[Theorem~1.3.1]{Tretter08}.  Moreover, the resolvent can be estimated
in terms of the distance to $W^2(S)$:
\begin{equation}\label{resolvestW2a}
  \|(S-\lambda)^{-1}\| \le \frac{\|S\|+|\lambda|}{[\dist(\lambda,W^2(S))]^2},
  \qquad \lambda\notin\ol{W^2(S)};
\end{equation}
see \cite[Theorem~1.4.1]{Tretter08}.  If $\ol{W^2(S)}=F_1\cup F_2$
with disjoint non-empty closed sets $F_1$ and $F_2$, then
\begin{equation}\label{resolvestW2b}
  \|(S-\lambda)^{-1}\| \le \frac{\|S\|+|\lambda|}{\dist(\lambda,F_1)\dist(\lambda,F_2)},
  \qquad \lambda\notin\ol{W^2(S)};
\end{equation}
see \cite[Theorem~1.4.5]{Tretter08}.

The quadratic numerical range definition can be easily extended to
unbounded block operator matrices, restricting the vectors $f$ and $g$ in \eqref{defSfg}
to the proper domains $(\dom A)\cap(\dom C)$ and $(\dom B)\cap(\dom D)$, respectively.
For details see \cite[Definition~2.5.1]{Tretter08}.

\medskip
In the following we recall the definition of the Schur complements of a block operator matrix,
which are powerful tools to study the spectrum and spectral properties.
Let $A$ and $D$ be closed operators in $\calH_+$ and $\calH_-$, respectively,
$B\in L(\calH_-,\calH_+)$, and $C\in L(\calH_+,\calH_-)$.  For the block operator matrix
\[
  S = \mymatrix{ A & B \\[0.5ex] C & D }, \quad\dom S = \dom A\oplus\dom D,
\]
the first and second Schur complements of $S$ are defined by:
\begin{align}\label{eq:Schurcomp}
  S_1(\la) &\defeq A - \la - B(D-\la)^{-1}C,\qquad\la\in\rho(D), \\
  S_2(\la) &\defeq D - \la - C(A-\la)^{-1}B,\qquad\la\in\rho(A).
\end{align}
These are analytic operator functions defined on the resolvent sets of $D$ and $A$, respectively.

In the %proof of Theorem \ref{t:new}
next sections we shall make use of the following
auxiliary result from \cite{N89}; see also \cite[Theorem 2.3.3]{Tretter08}. % and \cite[Lemma 2.1]{DMS15}.

% -------------------------------------------------------------------
\begin{lem}[{\cite[Theorem 2.4]{N89}}]\label{l:schurli}
Let $A$ and $D$ be closed operators in $\calH_+$ and $\calH_-$, respectively,
let $B\in L(\calH_-,\calH_+)$ and $C\in L(\calH_+,\calH_-)$, and consider the block operator matrix
\[
  S = \mymatrix{ A & B \\[0.5ex] C & D },\quad\dom S = \dom A\oplus\dom D.
\]
%and define the Schur complements:
%\begin{align*}
%S_1(\la) &\defeq A - \la - B(D-\la)^{-1}C,\qquad\la\in\rho(D),\\
%S_2(\la) &\defeq D - \la - C(A-\la)^{-1}B,\qquad\la\in\rho(A).
%\end{align*}
Then the following statements hold:
\begin{myenum}
\item  For $\la\in\rho(D)$ one has $\la\in\sigma(S)$ if and only if\, $0\in\sigma(S_1(\la))$.
\item  For $\la\in\rho(A)$ one has $\la\in\sigma(S)$ if and only if\, $0\in\sigma(S_2(\la))$.
\end{myenum}
Moreover, if $\la\in\rho(D)\cap\rho(S)$, then
\[
  (S - \la)^{-1}
  = \bmat{I}{0}{-(D-\la)^{-1}C}{I}  \bmat{S_1(\la)^{-1}}{0}{0}{(D-\la)^{-1}}
  \bmat{I}{-B(D-\la)^{-1}}{0}{I}.
  %\bmat{S_1(\la)^{-1}}{-S_1(\la)^{-1}B(D-\la)^{-1}}{-(D-\la)^{-1}CS_1(\la)^{-1}}{(D-\la)^{-1}[\,I+CS_1(\la)^{-1}B(D-\la)^{-1}]}.
\]
\end{lem}

%\medskip
\medskip

% *******************************************************************
\section{Sharp enclosures for the quadratic numerical range}\label{specinc}
% *******************************************************************

\noindent
Let $S$ be as in \eqref{PlazaMalvinas} with bounded operators
\[
  A \in L(\cH_+), \quad B \in L(\cH_-,\cH_+), \quad D \in L(\cH_-),
\]
where $A$ and $D$ are self-adjoint in the Hilbert spaces $\cH_+$ and $\cH_-$,
respectively.  Hence, the numerical ranges $W(A)$ and $W(D)$ are real intervals.
We introduce the following numbers, which are used for the description
of the enclosures that are proved below:
\begin{alignat}{2}
  a_- &\defeq \inf\, W(A),  & a_+ &\defeq \sup\, W(A), \label{defconstapm}\\
  d_- &\defeq \inf\, W(D),  & d_+ &\defeq \sup\, W(D), \label{defconstdpm}\\
%  c_- &\defeq \max\{a_-,d_-\}, \qquad & c_+ &\defeq \min\{a_+,d_+\}, \label{defconst3}\\
  m_- &\defeq \frac{a_-+d_-}{2}\,,  & m_+ &\defeq \frac{a_++d_+}{2}, \label{defconstmpm}\\
  c \defeq \frac{1}{2}\bigl[\min\{a_+,d_+\}+&\max\{a_-,d_-\}\bigr]  \;\;\;\;\text{and}&
  \ell &\defeq \frac 1 2\dist\bigl(\ol{W(A)},\ol{W(D)}\bigr)
	\label{defconstcl}
%  \ell &= \frac 1 2\dist\bigl(\ol{W(A)},\ol{W(D)}\bigr). \label{defconstl}
\end{alignat}
If $\ol{W(A)}\cap\ol{W(D)}=\varnothing$, then $c$ is the midpoint
of the gap between $\ol{W(A)}=[a_-,a_+]$ and $\ol{W(D)}=[d_-,d_+]$,
and $\ell$ is half the length of the gap;
e.g.\ if $d_+<a_-$, then %$c=\frac{1}{2}(d_++a_-)$ and $\ell=a_--d_+$.
\[
  c=\tfrac{1}{2}(d_++a_-)\quad \text{and}\quad \ell =\tfrac{1}{2}(a_--d_+).
\]

The following proposition contains enclosures of the closure of the
quadratic numerical range of operators of the form \eqref{PlazaMalvinas}.
Due to \eqref{W2specincl} these yield also enclosures for the spectrum.
%It improves similar results from \cite[Theorem~4.2]{LMM97}; see also
%\cite[Theorem~2.1]{LLMT05}, \cite[Proposition~2.6.8]{Tretter08},
%\cite[Theorem~5.5]{T09} and \cite[Theorem~3.5]{BPT13};
%see Remark~\ref{re:encl_lit}.
Note that the enclosures in Proposition~\ref{pr:spec_incl}
depend only on $W(A)$, $W(D)$ and $\|B\|$.
They are illustrated in Figures~\ref{fig:encl_nonsep_B_small}--\ref{fig:encl_sep_B_large} below.

% -------------------------------------------------------------------
\begin{prop}\label{pr:spec_incl}
Given a Hilbert space $\cH=\cH_+\oplus\cH_-$, consider the block operator matrix
\begin{equation}\label{bom_Jsa}
  S \defeq \mymatrix{ A & B \\[0.5ex] -B^* & D }
\end{equation}
where $B\in L(\cH_-,\cH_+)$, and $A$ and $D$ are bounded self-adjoint operators
in $\cH_+$ and $\cH_-$, respectively.
Further, let the constants $a_\pm$, $d_\pm$, $m_\pm$, $c$ and $\ell$
be as in \eqref{defconstapm}--\eqref{defconstcl}.  Then
\begin{align}
 % \sigma(S)\cap\dR &\subseteq
  \ol{W^2(S)}\cap\dR &\subseteq
  \bigl[\min\{a_-,d_-\},\max\{a_+,d_+\}\bigr],
  \label{realspecincl}\\[0.5ex]
%  \sigma(S)\setminus\dR &\subseteq
  \ol{W^2(S)}\setminus\dR &\subseteq
  \mathcal{B}_{\|B\|}\bigl([a_-,a_+]\bigr)\,\cap\,
  \mathcal{B}_{\|B\|}\bigl([d_-,d_+]\bigr)\,\cap\,\{z\in\C : m_- \le \RE z \le m_+\}.
  \label{nonrealspecincl}
\end{align}

Moreover, in the special case where
\begin{equation}\label{e:auseinander}
  \ol{W(A)}\cap\ol{W(D)} = \varnothing
  \qquad\text{and}\qquad
  \|B\|\le \ell, %\frac{\dist\bigl(\ol{W(A)},\ol{W(D)}\bigr)}{2}  = \ell,
\end{equation}
we have
\begin{equation}\label{specincla}
%  \sigma(S) \subseteq
  \ol{W^2(S)} \subseteq \Biggl[\min\{a_-,d_-\},\;
  c - \sqrt{\ell^2-\|B\|^2}\,\Biggr] \cup \Biggl[c + \sqrt{\ell^2-\|B\|^2},\;
  \max\{a_+,d_+\}\Biggr];
\end{equation}
in particular, both the quadratic numerical range and the spectrum of $S$ are real.
\end{prop}

\begin{proof}
Inclusion \eqref{realspecincl} follows from \cite[Proposition~1.2.6]{Tretter08}.
It is sufficient to show \eqref{nonrealspecincl} and \eqref{specincla}
without the closures on the left-hand sides because the right-hand sides are closed sets.
Let $z\in W^2(S)$.
Then there exist $f\in\cH_+$ and $g\in\cH_-$ with $\|f\|=\|g\|=1$
such that $z$ is an eigenvalue of the matrix $S_{f,g} = \smallmat \alpha\beta{-\ol\beta}\delta$
in \eqref{defSfg} with $\alpha = (Af,f)\in[a_-,a_+]$, $\delta = (Dg,g)\in[d_-,d_+]$
and $\beta = (Bg,f)$, which satisfies $|\beta|\le\|B\|$.
By \cite[Theorem~2.1]{LLMT05} and \cite[Theorem~3.5]{BPT13} the inclusion \eqref{nonrealspecincl}
holds for $\ol{W^2(S)}$ replaced by $\sigma(S)$.
Applying this to the matrix $S_{f,g}$ we obtain that in the case when $z$ is non-real,
$z$ is contained in
\[
  \calB_{|\beta|}(\{\alpha\})\cap\calB_{|\beta|}(\{\delta\})\cap\{z\in\C : m_- \le \RE z \le m_+\}.
\]
and hence in the right-hand side of \eqref{nonrealspecincl}.

It remains to show \eqref{specincla}.
We assume without loss of generality that $d_+<a_-$, in which case $c=\frac{1}{2}(a_-+d_+)$ and
$\ell=\frac 1 2(a_--d_+)$.
By \eqref{nonrealspecincl}, $z\in\R$, and hence $z=z_+$ or $z=z_-$ where
\[
  z_\pm \defeq \frac{\alpha+\delta}{2}
  \pm \sqrt{\Bigl(\frac{\alpha-\delta}{2}\Bigr)^2-|\beta|^2}.
\]
It is easy to see that $z_+$ is increasing in $\alpha$ and decreasing in $\delta$.
Therefore
\[
  z_+ \ge \frac{a_-+d_+}{2}+\sqrt{\Bigl(\frac{a_--d_+}{2}\Bigr)^2-|\beta|^2}
  \ge c + \sqrt{\ell^2-\|B\|^2}
\]
and, similarly, $z_-\le c - \sqrt{\ell^2-\|B\|^2}$. Together with \eqref{realspecincl}, this shows the inclusion \eqref{specincla}.
\end{proof}

% -------------------------------------------------------------------
\begin{rem}\label{re:encl_lit}
\begin{myenuma}
\item
Parts of Proposition \ref{pr:spec_incl} are known:
inclusion \eqref{realspecincl} is from \cite[Proposition~1.2.6]{Tretter08}.

\item
Inclusion \eqref{nonrealspecincl} is an improvement of a similar result in
\cite[Proposition~1.2.6]{Tretter08}.
To be more precise, in \cite[Proposition~1.2.6]{Tretter08} the following
inclusion is proved, when $\|B\|>\ell$,
\begin{equation}\label{nonrealinclTretter}
  \ol{W^2(S)}\setminus\R
  \subseteq \biggl\{z\in\dC: |\Im z| \le \sqrt{\|B\|^2-\ell^2}\biggr\}
  \cap \bigl\{z\in\C : m_- \le \RE z \le m_+\bigr\};
\end{equation}
cf.\ also \cite[Proposition~1.3.9]{Tretter08}.
It is easy to see that the right-hand side of \eqref{nonrealspecincl}
is contained in the right-hand side of \eqref{nonrealinclTretter},
and in most cases the inclusion is strict;
see also Figures~\ref{fig:encl_nonsep_B_small}--\ref{fig:encl_sep_B_large} below.

%Obviously, the above bound $\sqrt{\|B\|^2-\ell^2}$
%for the non-real part of the quadratic numerical range is smaller than $\|B\|$. However, depending
%on the location of the intervals $[a_-,a_+]$ and $[d_-,d_+]$, \eqref{nonrealspecincl}, in some cases,
%will give a sharper estimate for $ \ol{W^2(S)}\setminus\R$.

\item
Inclusion \eqref{specincla} improves \cite[Proposition~1.2.6]{Tretter08} significantly,
namely, the former implies that the interval
\begin{equation*}
\begin{aligned}
 \biggl(
  c - \sqrt{\ell^2-\|B\|^2}\, ,c + \sqrt{\ell^2-\|B\|^2}\biggr)
\end{aligned}
\end{equation*}
has empty intersection with $\ol{W^2(S)}$ if $\|B\|<\ell$.  %This result is new.
A similar result for the spectrum of $S$ (which is, in general, a smaller set)
can be found in \cite[Theorem~4.2]{CT16}
and, in a somehow different form, in \cite[Theorem~5.8]{AMS09} and \cite[Theorem~5.4]{AMT10}.

\item
For enclosures for the quadratic numerical range and the spectrum
where all entries $A$, $B$, and $D$ are allowed to be unbounded
see \cite[Proposition~4.10 and Theorem~4.13]{LS17}.
\end{myenuma}
\end{rem}

Figures~\ref{fig:encl_nonsep_B_small}--\ref{fig:encl_sep_B_large} below show
the enclosures for
%$\sigma(S)$ and
$\ol{W^2(S)}$ from Proposition~\ref{pr:spec_incl}.
In Figures~\ref{fig:encl_nonsep_B_small} and \ref{fig:encl_nonsep_B_large}
the situation where $\ol{W(A)}\cap\ol{W(D)}\ne\varnothing$ is considered.
%For small $\|B\|$ the non-real spectrum is contained in
If $\|B\|$ is less than or equal to
\[
  \tau \defeq \min\big\{m_+-\min\{a_+,d_+\}, \,\max\{a_-,d_-\}-m_-\big\},
\]
the non-real spectrum is contained in
$\mathcal{B}_{\|B\|}([a_-,a_+])\cap \mathcal{B}_{\|B\|}([d_-,d_+])=\mathcal{B}_{\|B\|}(\ol{W(A)}\cap\ol{W(D)})$.
When $\|B\|>\tau$ then the enclosure in $\{z\in\dC:m_-\le\Re z\le m_+\}$,
the third set on the right-hand side of \eqref{nonrealspecincl},
has to be taken into account as well.

The case when there is a gap between $\ol{W(A)}$ and $\ol{W(D)}$ is considered
in Figures~\ref{fig:encl_sep_B_small} and \ref{fig:encl_sep_B_large}.
When $\|B\|$ is small, then
%$\sigma(S)$ and
$\ol{W^2(S)}$ is contained
in the union of the two real intervals on the right-hand side of \eqref{specincla}.
When $\|B\|$ is larger, then the spectrum may be non-real, and the
right-hand sides of \eqref{realspecincl} and \eqref{nonrealspecincl}
have to be used.

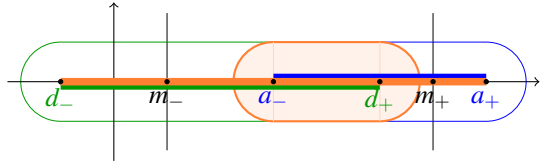
\begin{figure}[H]
\begin{center}
\begin{tikzpicture}[scale=.7]
% a_-=3, a_+=7, d_-=-1, d_+=5, \|B\|=0.75
  \draw[black, ultra thin] (1,-1.3) -- (1,1.3);
  \draw[black, ultra thin] (6,-1.3) -- (6,1.3);
  \draw[blue, ultra thin] (5,0.75) -- (7,0.75) arc (90:-90:0.75) -- (5,-0.75);
  \draw[mygreen, ultra thin] (3,0.75) -- (-1,0.75) arc (90:270:0.75) -- (3,-0.75);
  \filldraw[color=enclcol, thick, fill=enclcol!10] (5,0.75) -- (3,0.75) arc (90:270:0.75)
  -- (5,-0.75) arc (-90:90:0.75);
  \draw[color=enclcol!30, thin] (3,0.75) -- (3,-0.75);
  \draw[color=enclcol!30, thin] (5,0.75) -- (5,-0.75);
  \draw[->] (-2,0) -- (8,0);
  \draw[->] (0,-1.5) -- (0,1.5);
  \draw[color=mygreen, line width=1mm] (-1,-0.08) -- (5,-0.08);
  \draw[color=blue, line width=1mm] (3,0.08) -- (7,0.08);
  \draw[color=enclcol, line width=1mm] (-1,0) -- (7,0);
  \filldraw[black] (-1,0) circle (0.04);
  \filldraw[black] (1,0) circle (0.04);
  \filldraw[black] (3,0) circle (0.04);
  \filldraw[black] (5,0) circle (0.04);
  \filldraw[black] (6,0) circle (0.04);
  \filldraw[black] (7,0) circle (0.04);
  \node[color=mygreen] at (-1,-0.35) {$d_-$};
  \node at (1,-0.35) {$m_-$};
  \node[color=blue] at (3,-0.35) {$a_-$};
  \node[color=mygreen] at (5,-0.35) {$d_+$};
  \node at (6,-0.35) {$m_+$};
  \node[color=blue] at (7,-0.35) {$a_+$};
\end{tikzpicture}
\caption{The region (indicated in orange) given by the union of the sets
on the right-hand sides of \eqref{realspecincl} and \eqref{nonrealspecincl}
that contains
%$\sigma(S)$ and
$\ol{W^2(S)}$,
when $\ol{W(A)}$ and $\ol{W(D)}$ overlap and $\|B\|< \tau$.}
\label{fig:encl_nonsep_B_small}
\end{center}
\end{figure}

\begin{figure}[H]
\begin{center}
\begin{tikzpicture}[scale=.7]
% a_-=3, a_+=7, d_-=-1, d_+=5, \|B\|=1.75
  \draw[black, ultra thin] (1,-1.9) -- (1,1.9);
  \draw[black, ultra thin] (6,-1.9) -- (6,1.9);
  \draw[blue, ultra thin] (5,1.75) -- (7,1.75) arc (90:-90:1.75) -- (5,-1.75);
  \draw[mygreen, ultra thin] (3,1.75) -- (-1,1.75) arc (90:270:1.75) -- (3,-1.75);
  \draw[mygreen, ultra thin] (6,1.436) arc (55:-55:1.75);
  \filldraw[color=enclcol, thick, fill=enclcol!10] (5,1.75) -- (3,1.75) arc (90:270:1.75)
  -- (5,-1.75) arc (-90:-55:1.75) -- (6,1.436) arc (55:90:1.75);
  \draw[color=enclcol!30, thin] (3,1.75) -- (3,-1.75);
  \draw[color=enclcol!30, thin] (5,1.75) -- (5,-1.75);
  \draw[->] (-2.9,0) -- (9,0);
  \draw[->] (0,-2.1) -- (0,2.1);
  \draw[color=mygreen, line width=1mm] (-1,-0.08) -- (5,-0.08);
  \draw[color=blue, line width=1mm] (3,0.08) -- (7,0.08);
  \draw[color=enclcol, line width=1mm] (-1,0) -- (7,0);
  \filldraw[black] (-1,0) circle (0.04);
  \filldraw[black] (1,0) circle (0.04);
  \filldraw[black] (3,0) circle (0.04);
  \filldraw[black] (5,0) circle (0.04);
  \filldraw[black] (6,0) circle (0.04);
  \filldraw[black] (7,0) circle (0.04);
  \node[color=mygreen] at (-1,-0.35) {$d_-$};
  \node at (1,-0.35) {$m_-$};
  \node[color=blue] at (3,-0.35) {$a_-$};
  \node[color=mygreen] at (5,-0.35) {$d_+$};
  \node at (6,-0.35) {$m_+$};
  \node[color=blue] at (7,-0.35) {$a_+$};
\end{tikzpicture}
\caption{The region (indicated in orange) given by the union of the sets
on the right-hand sides of \eqref{realspecincl} and \eqref{nonrealspecincl}
that contains  $\ol{W^2(S)}$,
when $\ol{W(A)}$ and $\ol{W(D)}$ overlap and $\|B\|>\tau$.}
\label{fig:encl_nonsep_B_large}
\end{center}
\end{figure}
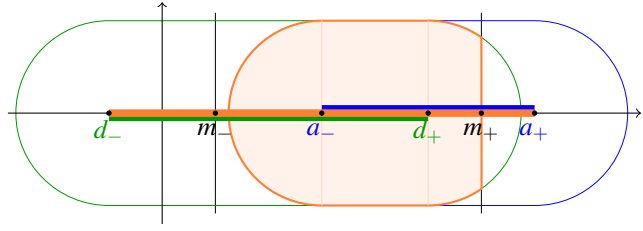

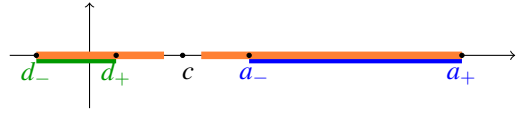
\begin{figure}[H]
\begin{center}
\begin{tikzpicture}[scale=.7]
% a_-=3, a_+=7, d_-=-1, d_+=0.5, \|B\|=6/5=1.2
  \draw[->] (-1.5,0) -- (8,0);
  \draw[->] (0,-1) -- (0,1);
  \draw[color=mygreen, line width=1mm] (-1,-0.08) -- (0.5,-0.08);
  \draw[color=blue, line width=1mm] (3,-0.08) -- (7,-0.08);
  \draw[color=enclcol, line width=1mm] (-1,0) -- (1.4,0);
  \draw[color=enclcol, line width=1mm] (2.1,0) -- (7,0);
  \filldraw[black] (-1,0) circle (0.04);
  \filldraw[black] (0.5,0) circle (0.04);
%  \filldraw[black] (1,0) circle (0.04);
  \filldraw[black] (1.75,0) circle (0.04);
  \filldraw[black] (3,0) circle (0.04);
%  \filldraw[black] (3.75,0) circle (0.04);
  \filldraw[black] (7,0) circle (0.04);
  \node[color=mygreen] at (-1,-0.35) {$d_-$};
  \node[color=mygreen] at (0.5,-0.35) {$d_+$};
%  \node at (1.1,-0.4) {$m_-$};
  \node at (1.85,-0.35) {$c$};
%  \node at (3.75,-0.4) {$m_+$};
  \node[color=blue] at (3.1,-0.4) {$a_-$};
  \node[color=blue] at (7,-0.4) {$a_+$};
\end{tikzpicture}
\caption{The two intervals (indicated in orange)
on the right-hand side of \eqref{specincla} whose union
contains  $\ol{W^2(S)}$,
when $\ol{W(A)}$ and $\ol{W(D)}$ are separated and $\|B\|<\tau$.}
\label{fig:encl_sep_B_small}
\end{center}
\end{figure}

\begin{figure}[H]
\begin{center}
\begin{tikzpicture}[scale=.75]
% a_-=3, a_+=7, d_-=-1, d_+=0.5, \|B\|=2.25
  \draw[blue, ultra thin] (1.75,1.87) arc (124:90:2.25) -- (7,2.25)
  arc (90:-90:2.25) -- (3,-2.25) arc (-90:-124:2.25);
  \draw[blue, ultra thin] (1,1.03) arc (152:208:2.25);
  \draw[mygreen, ultra thin] (1.75,1.87) arc (56:90:2.25) -- (-1,2.25)
  arc (90:270:2.25) -- (0.5,-2.25) arc (-90:-56:2.25);
  \draw[black, ultra thin] (1,-2.6) -- (1,2.6);
  \draw[black, ultra thin] (3.75,-2.6) -- (3.75,2.6);
  \filldraw[color=enclcol, thick, fill=enclcol!10] (1.75,1.87) arc (124:152:2.25)
  -- (1,-1.03) arc (207:236:2.25) arc (-56:56:2.25);
  \draw[color=enclcol!30, thin] (1.75,1.87) -- (1.75,-1.87);
  \draw[->] (-3.5,0) -- (9.5,0);
  \draw[->] (0,-3) -- (0,3);
  \draw[color=mygreen, line width=1mm] (-1,-0.08) -- (0.5,-0.08);
  \draw[color=blue, line width=1mm] (3,-0.08) -- (7,-0.08);
  \draw[color=enclcol, line width=1mm] (-1,0) -- (7,0);
  \filldraw[black] (-1,0) circle (0.04);
  \filldraw[black] (0.5,0) circle (0.04);
  \filldraw[black] (1,0) circle (0.04);
  \filldraw[black] (1.75,0) circle (0.04);
  \filldraw[black] (3,0) circle (0.04);
  \filldraw[black] (3.75,0) circle (0.04);
  \filldraw[black] (7,0) circle (0.04);
  \node[color=mygreen] at (-1,-0.35) {$d_-$};
  \node[color=mygreen] at (0.5,-0.35) {$d_+$};
  \node at (1.1,-0.4) {$m_-$};
  \node at (1.85,-0.35) {$c$};
  \node at (3.75,-0.4) {$m_+$};
  \node[color=blue] at (3.1,-0.4) {$a_-$};
  \node[color=blue] at (7,-0.4) {$a_+$};
\end{tikzpicture}
\caption{The region (indicated in orange) given by the union of the sets
on the right-hand sides of \eqref{realspecincl} and \eqref{nonrealspecincl}
that contains $\ol{W^2(S)}$,
when $\ol{W(A)}$ and $\ol{W(D)}$ are separated and $\|B\|>\tau$.}
\label{fig:encl_sep_B_large}
\end{center}
\end{figure}
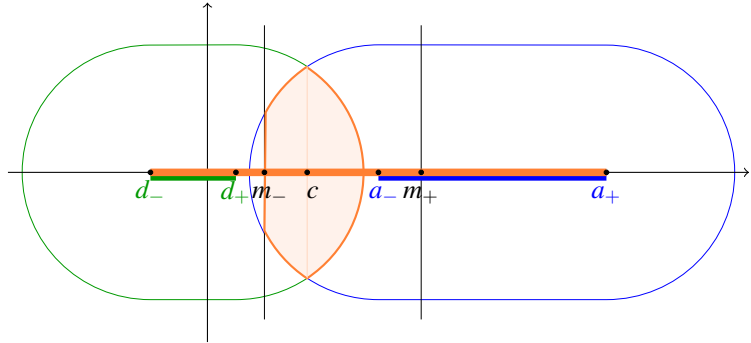

The next theorem shows that Proposition~\ref{pr:spec_incl} is sharp in the sense that given $[a_-,a_+]=\ol{W(A)}$, $[d_-,d_+]=\ol{W(D)}$
and $\|B\|$, the enclosures for the spectrum and the quadratic numerical range
of $S$ cannot be improved, i.e.\ an operator $S$ is constructed for which
equality holds in \eqref{realspecincl}
and \eqref{nonrealspecincl}, and in \eqref{specincla} if \eqref{e:auseinander}
is satisfied.

% -------------------------------------------------------------------
\begin{thm}\label{th:encl_sharp}
Let $a_+,a_-,d_+,d_-,b\in\dR$ such that $a_-\le a_+$, $d_-\le d_+$, and $b>0$.
Then there exist separable Hilbert spaces $\cH_\pm$, self-adjoint operators $A$ and $D$
in $\cH_+$ and $\cH_-$, respectively, and $B\in L(\cH_-,\cH_+)$ such that
\[
  \ol{W(A)} = [a_-,a_+], \qquad \ol{W(D)} = [d_-,d_+], \qquad \|B\| = b,
\]
and
\textup{(}with the notation from \eqref{defconstmpm}--\eqref{defconstcl}\textup{)}
the operator $S$ from \eqref{bom_Jsa} satisfies
\begin{equation}\label{equspecencl1}
\begin{aligned}
  \sigma(S) = \ol{W^2(S)} &= \Bigl(\mathcal{B}_b\bigl([a_-,a_+]\bigr) \cap \mathcal{B}_b\bigl([d_-,d_+]\bigr)
  \cap \{z\in\dC: m_-\le \Re z \le m_+\}\Bigr) \\[0.5ex]
  &\qquad \cup \bigl[\min\{a_-,d_-\},\max\{a_+,d_+\}\bigr]
\end{aligned}
\end{equation}
if $b>\ell$
and
\begin{equation}\label{equspecencl2}
  \sigma(S) = \ol{W^2(S)} = \Biggl[\min\{a_-,d_-\},\;
  c - \sqrt{\frac{\ell^2}{4}-b^2}\,\Biggr]
  \cup \Biggl[c + \sqrt{\frac{\ell^2}{4}-b^2},\;
  \max\{a_+,d_+\}\Biggr]
\end{equation}
if $b\le \ell$.
\end{thm}

\begin{proof}
Let $\cH_+=\cH_-=\ell^2$ and define the operators
\[
  A = \diag(a_1,a_2,\ldots), \qquad B = \diag(b_1,b_2,\ldots), \qquad
  D = \diag(d_1,d_2,\ldots)
\]
with numbers
\begin{equation}\label{anbndn_incl}
  a_n \in [a_-,a_+], \quad
  b_n \in [0,b], \quad
  d_n \in [d_-,d_+],
\end{equation}
%$a_n,b_n,d_n\in\dR$,
$n\in\dN$, which are chosen later.
Let $z_n=x_n+iy_n$, $n\in\dN$, be such that $\{z_n:n\in\dN\}$
is a dense subset of the right-hand sides of \eqref{equspecencl1}
or \eqref{equspecencl2}, respectively.
%\[
%  \sigma(S) = B_{\|B\|}\bigl([a_-,a_+]\bigr) \cap B_{\|B\}}\bigl([d_-,d_+]\bigr)
%  \cap \{z\in\dC: m_-\le \Re z \le m_+\} \cap \dC^+.
%\]
Below we construct $a_n,b_n,d_n$
such that
\begin{equation}\label{znandnbn}
  z_n=w_+(a_n,b_n,d_n) \qquad\text{or}\qquad z_n=w_-(a_n,b_n,d_n),
\end{equation}
where
\begin{equation}\label{wpmandnbn}
  w_\pm(a_n,b_n,d_n) \defeq \frac{a_n+d_n}{2}
  \pm \sqrt{\Bigl(\frac{a_n-d_n}{2}\Bigr)^2-|b_n|^2}.
\end{equation}
Since then $z_n$ is an eigenvalue of $S$ and $\sigma(S)$ is closed, this,
together with the enclosures in Proposition~\ref{pr:spec_incl},
shows equality in \eqref{equspecencl1} and \eqref{equspecencl2}.

Let us first consider the case when $z_n\notin\dR$.
Then $z_n$ is in the right-hand side of \eqref{equspecencl1}.
If $x_n\in[a_-,a_+]\cap[d_-,d_+]$, then set
\[
  a_n \defeq d_n \defeq x_n, \qquad b_n \defeq |y_n|.
\]
Clearly,
\[
  a_n\in[a_-,a_+], \quad d_n\in[d_-,d_+], \quad
  b_n = |y_n| = \dist(z_n,[a_-,a_+]) \le b,
\]
and \eqref{znandnbn} is satisfied.
Now assume that $x_n\notin[a_-,a_+]\cap[d_-,d_+]$.
Without loss of generality we can assume that
\begin{equation}\label{wlogequal}
  \dist\bigl(x_n,[a_-,a_+]\bigr) \le \dist\bigl(x_n,[d_-,d_+]\bigr),
\end{equation}
which implies that $x_n\notin[d_-,d_+]$.
Let us consider the case when $x_n < d_-$; the case $x_n > d_+$ is analogous.
Set
\[
  a_n = 2x_n-d_-, \qquad d_n = d_-, \qquad
  b_n = \sqrt{y_n^2+(x_n-d_-)^2}.
\]
Clearly, $d_n\in[d_-,d_+]$.  From $x_n\ge m_-$ we obtain that
\[
  a_n = 2x_n - d_- \ge 2m_- - d_- = 2\frac{a_-+d_-}{2} - d_- = a_-.
\]
If $x_n \le a_+$, then $a_n = x_n+(x_n-d_-) < x_n \le a_+$ and hence $a_n\in[a_-,a_+]$.
If $x_n > a_+$, then the inequality in \eqref{wlogequal}
is equivalent to $x_n-a_+ \le d_--x_n$, which implies that
\[
  a_n = 2x_n - d_- \le a_+;
\]
hence also in this case we have $a_n\in[a_-,a_+]$.
Moreover,
\[
  b_n = \sqrt{y_n^2+(x_n-d_-)^2}
  = |z_n-d_-| = \dist\bigl(z_n,[d_-,d_+]\bigr) \le b.
\]
It is easy to check that \eqref{znandnbn} is satisfied.

Next we consider the case when $z_n\in\dR$.
If $z_n\in[a_-,a_+]$, then choose $a_n=z_n$, $d_n$ arbitrary in $[d_-,d_+]$
and $b_n=0$.  Then
\[
  w_+(a_n,b_n,d_n) = \max\{z_n,d_n\}, \qquad w_-(a_n,b_n,d_n) = \min\{z_n,d_n\},
\]
and hence \eqref{znandnbn} holds.
The case when $z_n\in[d_-,d_+]$ is similar.
If $[a_-,a_+]\cap[d_-,d_+]\ne\varnothing$, then all cases of real $z_n$ are covered.
Finally, assume that $[a_-,a_+]\cap[d_-,d_+]=\varnothing$
and $z_n\notin[a_-,a_+]\cup[d_-,d_+]$.
Without loss of generality we can assume that $d_+<a_-$;
then $z_n\in(d_+,a_-)$.
Let us consider the case when $z_n \ge c = \frac{1}{2}(d_++a_-)$;
the other case is similar.  Set
\[
  a_n = a_-, \qquad d_n = d_+, \qquad
  b_n = \sqrt{\Bigl(\frac{\ell}{2}\Bigr)^2-(z_n-c)^2}.
\]
It is easy to check that $z_n=w_+(a_n,b_n,d_n)$.
If $b>\frac{1}{2}\dist([a_-,a_+],[d_-,d_+])=\frac{\ell}{2}$,
then $b_n\le \frac{\ell}{2}<b$.
If $b\le\frac{\ell}{2}$, then the form of the right-hand side
of \eqref{equspecencl2} implies that
\[
  z_n \ge c + \sqrt{\Bigl(\frac{\ell}{2}\Bigr)^2-b^2},
\]
which yields
\[
  b_n \le \sqrt{\Bigl(\frac{\ell}{2}\Bigr)^2 - \biggl(\Bigl(\frac{\ell}{2}\Bigr)^2-b^2\biggr)}
  = b.
\]

The relations in \eqref{anbndn_incl} imply that
%$a_n\in[a_-,a_+]$, $d_n\in[d_-,d_+]$ and $b_n\in[0,b]$,
the operators $A$, $B$ and $D$ are bounded with $\|B\|\le b$.
If we had $\|B\|<b$, then we would obtain a strictly smaller
enclosure for the spectrum from Proposition~\ref{pr:spec_incl},
which contradicts the already obtained equality
in \eqref{equspecencl1} or \eqref{equspecencl2}, respectively.
\end{proof}

% *******************************************************************
\section{Gershgorin-type enclosure for the spectrum of block operator matrices}
\label{smalll}
% *******************************************************************

\noindent
In this section we provide another spectral enclosure for the non-real spectrum
of the block operator matrix
\begin{equation}\label{bom_gersh_1}
  S = \mymatrix{ A & B \\[0.5ex] -B^* & D },\qquad\dom S = \dom A\oplus\dom D.
\end{equation}
As already indicated by \eqref{bom_gersh_1}, we also allow unbounded entries $A$ and $D$.
The operator $B$ remains bounded in our considerations.
The result has similarities with Gershgorin's circle theorem for matrices \cite{Ger31}
and block operator matrices \cite{RT18,Salas99,Tretter08,DMS15} since we show that the non-real spectrum
of the operator matrix $S$ is contained in the union of a family of closed balls,
centred along parts of the spectrum of the block $A$ in the diagonal of $S$ (see \eqref{e:general}).
To formulate the result, for a closed set $M\subseteq\R$ define the
following class of continuous functions:
\begin{equation}\label{defC+}
  C^+(M) = \Bigl\{f\in C(M) : f(t)>0 \;\;\text{for}\;\;t\in M,\;\sup_{t\in M}f(t) < \infty,\;
  \inf_{|t|\ge 1}\,|t|f(t) > 0\Bigr\}.
\end{equation}
The last two conditions obviously only matter if $M$ is unbounded.
If $M$ is compact, then $C^+(M)$ is the set of positive continuous functions on $M$.
Note that any positive constant function is contained in $C^+(M)$
and also $|t|^{-1}\in C^+(M)$ if $0\notin M$.

% -------------------------------------------------------------------
\begin{thm}\label{t:gersh_f}
Given a Hilbert space $\cH=\cH_+\oplus\cH_-$, consider the block operator matrix $S$
in \eqref{bom_gersh_1}, where $B\in L(\cH_-,\cH_+)$ and $A$ and $D$ are self-adjoint operators
in $\calH_+$ and $\calH_-$, respectively.
Then, for any $f\in C^+(\sigma(A))$ and $g\in C^+(\sigma(D))$ we have
\begin{equation}\label{e:general}
  \sigma(S)\setminus\R\,\subseteq\,
  \left(\bigcup_{t\in\sigma(A)}\mathcal{B}_{f(t)^{-1}\|f(A)B\|}(t)\right)\,\cap\,
  \left(\bigcup_{s\in\sigma(D)}\mathcal{B}_{g(s)^{-1}\|g(D)B^*\|}(s)\right).
\end{equation}
\end{thm}

\begin{rem}
If we set $f = g = 1$ in Theorem~\ref{t:gersh_f}, then the spectral inclusion \eqref{e:general} becomes
\begin{equation}\label{e:bpt}
  \sigma(S)\setminus\R\,\subseteq\,\mathcal{B}_{\|B\|}(\sigma(A))\,\cap\,\mathcal{B}_{\|B\|}(\sigma(D)),
\end{equation}
which was already proved in \cite[Theorem~3.5]{BPT13}.
\end{rem}

Theorem~\ref{t:gersh_f} will follow from Theorem \ref{t:new} below,
which is an improvement of \cite[Theorem~3.5]{BPT13}.
The spectral inclusion \eqref{e:bpt} means that a non-real point in the
spectrum of $S$ satisfies $\dist(\la,\sigma(A))\le\|B\|$ and $\dist(\la,\sigma(D))\le\|B\|$.
This is equivalent to $\|(A-\la)^{-1}\|\|B\|\ge 1$ and $\|(D-\la)^{-1}\|\|B^*\|\ge 1$.
Hence, the spectral enclosure given in \eqref{e:best} is sharper.

% -------------------------------------------------------------------
\begin{thm}\label{t:new}
Given a Hilbert space $\cH=\cH_+\oplus\cH_-$, consider the block operator matrix $S$
in \eqref{bom_gersh_1}, where $B\in L(\cH_-,\cH_+)$ and $A$ and $D$
are self-adjoint operators in $\calH_+$ and $\calH_-$, respectively. Then
\begin{equation}\label{e:best}
  \sigma(S)\setminus\R\,\subseteq\,\bigl\{\la\in\C\setminus\R : \|(A-\la)^{-1}B\|\ge 1\;
  \text{ and }\;\|(D-\la)^{-1}B^*\|\ge 1\bigr\}.
\end{equation}
%Let $\la\in\C\setminus\R$ and set
%\[
  %T_A(\lambda) \defeq (A-\la)^{-1}B
  %\qquad\text{and}\qquad
  %T_D(\lambda) \defeq (D-\la)^{-1}B^*;
%\]
Moreover, given $\la\in\C\setminus\R$
then
\begin{alignat}{2}
  \|(S-\la)^{-1}\| &\le \frac{1 + \|(A-\la)^{-1}B\| + \|(A-\la)^{-1}B\|^2}{|\Im\la|\cdot(1-\|(A-\la)^{-1}B\|^2)}\qquad
  && \text{if $\|(A-\la)^{-1}B\| < 1$},\label{e:resolvent1}\\[1ex]
  \|(S-\la)^{-1}\| &\le \frac {1 + \|(D-\la)^{-1}B^*\| + \|(D-\la)^{-1}B^*\|^2}{|\Im\la|\cdot(1-\|(D-\la)^{-1}B^*\|^2)}\qquad
  && \text{if $\|(D-\la)^{-1}B^*\| < 1$}.\label{e:resolvent2}
  %\|(S-\la)^{-1}\| &\le \frac{1 + \|T_A(\la)\| + \|T_A(\la)\|^2}{|\Im\la|\cdot(1-\|T_A(\la)\|^2)}\qquad
  %&& \text{if $\|T_A(\la)\| < 1$},\label{e:resolvent1}\\[1ex]
  %\|(S-\la)^{-1}\| &\le \frac {1 + \|T_D(\la)\| + \|T_D(\la)\|^2}{|\Im\la|\cdot(1-\|T_D(\la)\|^2)}\qquad
  %&& \text{if $\|T_D(\la)\| < 1$}.\label{e:resolvent2}
\end{alignat}
\end{thm}

\bigskip

Before we prove Theorem~\ref{t:new}, we provide a couple of remarks and an example.

\begin{rem}
(a) The same conclusions as in Theorem \ref{t:new} hold, if we drop the boundedness assumption on $B$
and, instead, assume that $B$ is $D$-bounded with $D$-bound less than one and $B^*$ is $A$-bounded
with $A$-bound less than one.  The arguments in the proof are essentially the same.

\medskip
(b) Note that both $\sigma(S)\setminus\R$ and the right-hand side of \eqref{e:best}
are sets which are symmetric with respect to the real axis.

\medskip
(c) There is another spectral enclosure for the operator matrix $S$ that results
from a relatively simple argument (see \cite[Theorem~1.1]{DMS15} or \cite[Lemma~5.2\,(ii)]{AMT10}).
Consider the second Schur complement $S_2(\la) = D-\la + B^*(A-\la)^{-1}B$ for $\la\in\C\setminus\R$.
Applying $(D-\la)^{-1}$ from the right and from the left, respectively, we obtain
\begin{align*}
  S_2(\la)(D-\la)^{-1} &= I + B^*(A-\la)^{-1}B(D-\la)^{-1}
  \qquad\text{and} \\
  (D-\la)^{-1}S_2(\la) &= I + (D-\la)^{-1}B^*(A-\la)^{-1}B|_{\dom D}.
\end{align*}
That is, if one of
\[
  N_S(\la) \defeq \|B^*(A-\la)^{-1}B(D-\la)^{-1}\|
  \qquad\text{or}\qquad
  N_S(\ol\la) = \|(D-\la)^{-1}B^*(A-\la)^{-1}B\|
\]
is less than $1$, then the Schur complement $S_2(\la)$ is boundedly invertible and so $\la\in\rho(S)$.
Hence,
\[
  \sigma(S)\,\subseteq\,\big\{\la\in\rho(D) : \;N_S(\la)\ge 1\;\text{ and }\; N_S(\ol\la)\ge 1\big\}.
\]
A similar reasoning applies to the first Schur complement $S_1(\la) = A-\la + B(D-\la)^{-1}B^*$ and gives
\[
  \sigma(S)\,\subseteq\,\big\{\la\in\rho(A): \;M_S(\la)\ge 1\;\text{ and }\; M_S(\ol\la)\ge 1\big\},
\]
where $M_S(\la) \defeq \|B(D-\la)^{-1}B^*(A-\la)^{-1}\|$.  This implies that
\begin{equation}\label{e:neumann}
  \sigma(S)\setminus\R\,\subseteq\,\big\{\la\in\C\setminus\R :
  \min\{N_S(\la),N_S(\ol\la),M_S(\la),M_S(\ol\la)\}\ge 1\big\}.
\end{equation}

\smallskip
(d) The spectral enclosures \eqref{e:best} and \eqref{e:neumann} are independent of each other,
meaning that, in general, none of the corresponding sets on the right-hand sides of the
two relations contains the other. Consequently, if we intersect the right-hand side of the
already known enclosure \eqref{e:neumann} with the new one \eqref{e:best},
we obtain a better bound for the non-real spectrum of $S$, as illustrated in
Example~\ref{ex:mathematica} below.
However, Example \ref{ex:gershgorin} shows a situation, where \eqref{e:best} is in fact strictly better
than \eqref{e:neumann}.

\medskip
(e) Both spectral enclosures \eqref{e:best} and \eqref{e:neumann} require complete knowledge
about the functions $\|(A-\cdot)^{-1}B\|$, $\|(D-\cdot)^{-1}B^*\|$, $N_S$ and $M_S$.
In contrast, Theorem \ref{t:gersh_f} basically only requires knowledge about $\sigma(A)$ and $\sigma(D)$
and is therefore better suited for computations.
\end{rem}

% -------------------------------------------------------------------
\begin{ex}\label{ex:mathematica}
We let $\calH_- = \calH_+ = \C^2$ and
\[
  A = \mymatrix{2 & 1+i \\[0.5ex] 1-i & -1}, \qquad
  D = \mymatrix{1 & 0 \\[0.5ex] 0 & -5}, \qquad
  B = \mymatrix{i & 1+\frac{i}{2} \\[0.5ex] -1-i & -\frac{2}{5}}.
\]
The four eigenvalues of $S\in\C^{4\times 4}$ are depicted as black dots in the figure below.
Note that two of them are real. They are (approximately) $-4.73166$, $2.38898$, $-0.328657\pm 1.03244\,i$.
The region from \eqref{e:neumann} is bounded by the three red curves,
while the two blue curves bound the region on the right-hand side of \eqref{e:best}.
The orange filled region is the intersection of the two enclosures.
\begin{figure}[ht]
\begin{center}
\includegraphics[width=9cm]{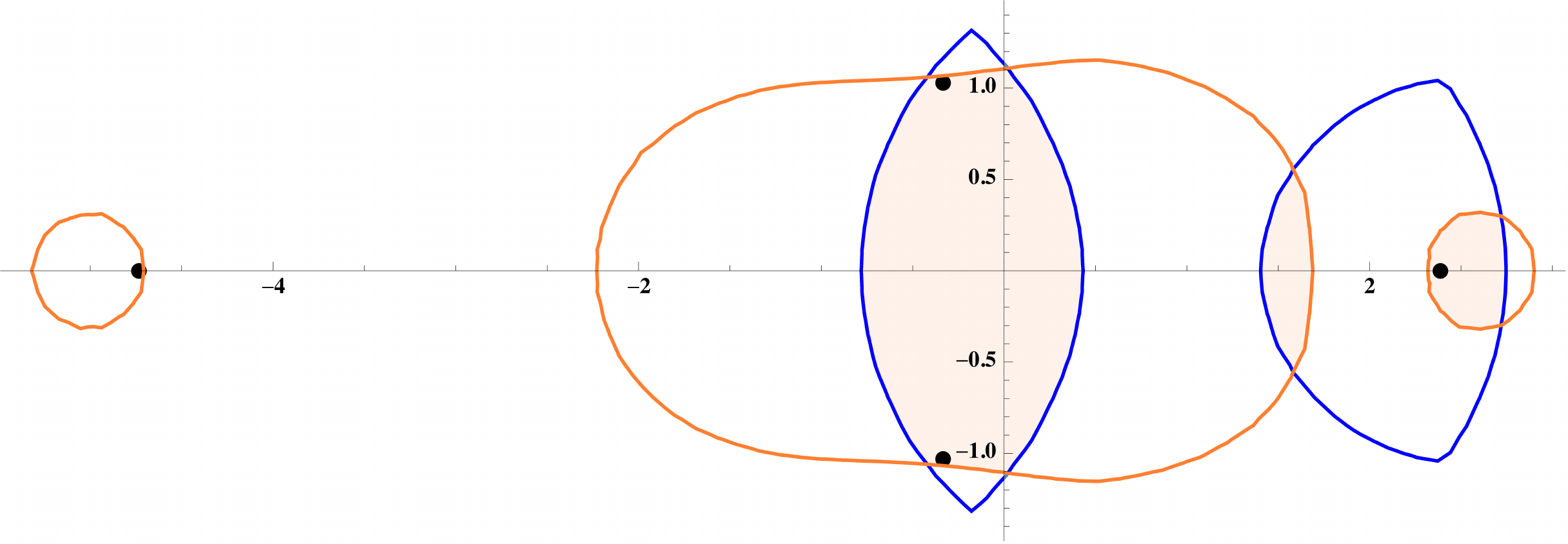}
\end{center}
\caption{The spectral enclosures \eqref{e:best}, bounded with blue curves,
and \eqref{e:neumann}, bounded with orange curves, for the matrix in Example~\ref{ex:mathematica}.}
\end{figure}
\end{ex}

\begin{proof}[Proof of Theorem~\rmref{t:new}]
We use the first Schur complement $S_1$ of the block operator matrix $S$ in \eqref{bom_gersh_1},  which is given by
\begin{align*}
  S_1(\lambda) = A - \lambda + B(D-\lambda)^{-1}B^*,\qquad\dom S_1(\la) = \dom A,
\end{align*}
for $\lambda\in\rho(D)$; see \eqref{eq:Schurcomp}.

For $\la\in\C\setminus\R$ we have $S_1(\la)^* = S_1(\ol\la)$ and, setting $T \defeq (D-\la)^{-1}B^*$ %T_D(\lambda) =
we obtain
\begin{align*}
  S_1(\la)-S_1(\ol\la)
  = \ol\la-\la + B\big[(D-\la)^{-1}-(D-\ol\la)^{-1}\big]B^* = (\la-\ol\la)(T^*T - I).
\end{align*}
If $\Im\la > 0$, then for arbitrary $h\in\dom A$ with $\|h\|=1$ we have
\begin{equation}\label{ImS1hhl0}
  \Im(S_1(\la)h,h)
  = \frac 1{2i}\big((S_1(\la)-S_1(\ol\la))h,h\big) = (\Im\la)\cdot(\|Th\|^2-1)\,\le\,(\Im\la)(\|T\|^2-1).
\end{equation}
In particular, if $\|T\|<1$, then $0\notin\ol{W(S_1(\lambda))}$ and
therefore $0\notin\sigma(S_1(\lambda))$; see Lemma \ref{l:kato}.
Now Lemma~\ref{l:schurli} implies that $\lambda\notin\sigma(S)$.
A similar reasoning applies to the case $\Im\la < 0$.
This proves that
\[
  \sigma(S)\setminus\R\,\subseteq\,\big\{\la\in\C\setminus\R : \|(D-\la)^{-1}B^*\|\ge 1\big\}.
\]
Applying the same arguments to the second Schur complement $S_2$, we obtain
\[
  \sigma(S)\setminus\R\,\subseteq\,\big\{\la\in\C\setminus\R : \|(A-\la)^{-1}B\|\ge 1\big\},
\]
which completes the proof of the inclusion \eqref{e:best}.

Note also that \eqref{ImS1hhl0} implies that
%Let us also remark that in case $\|T\|<1$ the above estimate shows
\begin{equation}\label{e:dist}
  \dist\big(0,\ol{W(S_1(\la))}\big)\,\ge\,|\Im\la|(1-\|T\|^2)
\end{equation}
if $\|T\|<1$.
Let us now prove the estimate \eqref{e:resolvent2} for the resolvent of $S$.
For this, let $\la\in\C\setminus\R$ such that $\|T\|<1$,
% T_D(\lambda) =
where $T = (D-\la)^{-1}B^*$ as above. By Lemma \ref{l:schurli} we have
\begin{equation}\label{e:resi}
  (S - \la)^{-1}
  = \bmat{I}{0}{(D-\la)^{-1}B^*}{I}  \bmat{S_1(\la)^{-1}}{0}{0}{(D-\la)^{-1}}
  \bmat{I}{-B(D-\la)^{-1}}{0}{I}.
  %\bmat{S_1(\la)^{-1}}{-S_1(\la)^{-1}B(D-\la)^{-1}}{(D-\la)^{-1}B^*S_1(\la)^{-1}}{(D-\la)^{-1}[\,I-B^*S_1(\la)^{-1}B(D-\la)^{-1}]}.
\end{equation}
Denote the first factor by $L$.  Then
\begin{align*}
  \|L\|^2 &= \|L^*L\|
  = \left\|\bmat{I}{T^*}{0}{I}\bmat{I}{0}{T}{I}\right\|
  = \left\|\bmat{I+T^*T}{T^*}{T}{I}\right\| \\[0.5ex]
  &\le\,\left\|\bmat{I}{0}{0}{I}\right\| + \left\|\bmat{T^*T}{0}{0}{0}\right\|
  + \left\|\bmat{0}{T^*}{T}{0}\right\| \\[0.5ex]
  &= 1 + \|T\|^2 + \|T\|.
\end{align*}
Since $(D-\la)^{-1}$ is normal, we have
\[
  \|B(D-\la)^{-1}\| = \|(D-\ol\la)^{-1}B^*\| = \|(D-\la)^{-1}B^*\| = \|T\|,
\]
which implies that for the last factor in \eqref{e:resi} we have
the same estimate as for the first one.
It remains to estimate the middle factor $M$ in \eqref{e:resi}.
To this end, note that Lemma \ref{l:kato} and \eqref{e:dist} yield
\[
  \|S_1(\la)^{-1}\| \le \dist\bigl(0,\ol{W(S_1(\la))}\bigr)^{-1}
  \le |\Im\la|^{-1}(1-\|T\|^2)^{-1}.
\]
Since $\|(D-\la)^{-1}\|\le|\Im\la|^{-1}$, we obtain
\[
  \|M\| = \max\big\{\|S_1(\la)^{-1}\|,\,\|(D-\la)^{-1}\|\big\}\,\le\,|\Im\la|^{-1}(1-\|T\|^2)^{-1}.
\]
Hence
\[
  \|(S-\la)^{-1}\| \le \frac{1+\|T\|+\|T\|^2}{|\Im\la|(1-\|T\|^2)}\,.
\]
The estimate \eqref{e:resolvent1} can be derived similarly by using the second Schur complement.
\end{proof}

% -------------------------------------------------------------------

%In the following we reformulate Theorem \ref{t:new} and obtain some infinite dimensional version of Gershgorin's circles theorem, see Theorem \ref{t:gersh_f2} below. This is done is such a way that the spectral enclosure from Theorem \ref{t:new} is expressed in terms of the spectral quantities of $A$, $B$ and $D$. This gives, compared with Theorem \ref{t:new}, a more intuitive insight into the location of the spectrum of $S$.

In the following we are going to show that Theorem \ref{t:gersh_f} is just a consequence of Theorem \ref{t:new}. However, since the enclosure in Theorem \ref{t:gersh_f} is expressed in terms of the spectral quantities of $A$ and $D$, compared with Theorem \ref{t:new} it gives a more intuitive and explicit insight into the location of the spectrum of $S$.

%The following lemma contains a uniform estimation of the distance between a sequence of continuous functions and its limit function. It is used in the proof of the next proposition.
\begin{lem}\label{l:helpy}
Let $c > 0$. Then there exists $C > 0$ \braces{depending on $c$} such that
\[
  \big|x^{1+1/n} - x\big| \le \frac{C}{n}
\]
for all $x\in [0,c]$ and all $n\in\N$.
\end{lem}

\begin{proof}
Let $x\in(0,c]$.  By the mean value theorem applied to the function $t\mapsto x^t$
there exists a $\xi\in(1,1+\frac{1}{n})$ such that
\[
  x^{1+1/n}-x = \frac{1}{n}x^\xi\log x.
%  \le \frac{1}{n}|\log x|\max\{x,x^{1+1/n}\}
\]
If $x\le 1$, then
\[
  \big|x^{1+1/n}-x\big| \le \frac{1}{n}x|\log x| \le \frac{1}{en}\,.
\]
If $c>1$ and $x\in(1,c]$, then
\[
  \big|x^{1+1/n}-x\big| \le \frac{1}{n}x^{1+1/n}\log x
  \le \frac{1}{n}x^2\log x \le \frac{c^2\log c}{n}\,.
\]
This proves the lemma.
%
%First we assume that $c\le 1$. Then for $x\in (0,c]$ we have
%\[
%|x^{1-1/n} - x| = \left|e^{(1-\frac 1n)\log x}-e^{\log x}\right|\,\le\,\frac 1n\cdot x^{1-\frac 1n}|\log x|\,\le\,\frac 1n\cdot \sqrt x|\log x|\,\le\,\frac{2}{ne}
%\]
%and
%\[
%|x^{1+1/n} - x| = \left|e^{(1+\frac 1n)\log x}-e^{\log x}\right|\,\le\,\frac 1n\cdot x|\log x|\,\le\,\frac{1}{ne}.
%\]
%If $c > 1$, then the above estimate holds for $x\in [0,1]$. For $x\in (1,c]$ we have
%$$
%|x^{1-1/n} - x| = \left|e^{(1-\frac 1n)\log x}-e^{\log x}\right|\,\le\,\frac 1n\cdot x\log x\,\le\,\frac{c\log c}n
%$$
%and
%\[
%|x^{1+1/n} - x| = \left|e^{(1+\frac 1n)\log x}-e^{\log x}\right|\,\le\,\frac 1n\cdot x^{1+1/n}\log x\,\le\,\frac{c^{3/2}\log c}n.
%\]
%This proves the lemma.
\end{proof}

Let $\calH_1$ and $\calH_2$ be Hilbert spaces, $T$ a self-adjoint operator in $\calH_1$ and $V\in L(\calH_2,\calH_1)$. Then by $\sigma_V(T)$ we denote the support of the positive operator-valued measure $V^*E_T(\cdot)V$, where $E_T$ stands for the spectral measure of $T$. Clearly, $\sigma_V(T)$ is a closed subset of $\sigma(T)$. It is compact if and only if $\ran V\subseteq E_T(\Delta)\calH$ for some bounded set $\Delta\subseteq\R$.

% -------------------------------------------------------------------
\begin{prop}\label{p:late}
Let $T$ be a self-adjoint operator in $\calH_1$ and $V\in L(\calH_2,\calH_1)$.
Then for $\la\in\C\setminus\R$ the following statements are equivalent:
\begin{myenuma}
\item $\|(T-\la)^{-1}V\|\ge 1$;
\item for all $f\in C^+(\sigma(T))$
we have $\la\in\bigcup_{t\in\sigma_V(T)}\mathcal{B}_{f(t)^{-1}\|f(T)V\|}(t)$.
\end{myenuma}
\end{prop}

\begin{proof}
(b)\,$\Sra$\,(a).
Let $f(t) \defeq |t-\la|^{-1}$, $t\in\sigma(T)$.
Then $f\in C^+(\sigma(T))$ and thus, by (b), we have $\la\in \mathcal{B}_{f(t_0)^{-1}\|f(T)V\|}(t_0)$
for some $t_0\in\sigma_V(T)$.  This means that
\[
  |t_0-\la|\,\le\,f(t_0)^{-1}\|f(T)V\| = |t_0-\la|\cdot\|(T-\la)^{-1}V\|,
\]
which is (a).

\smallskip
(a)\,$\Sra$\,(b). Let  $f\in C^+(\sigma(T))$.
It is obvious that $f$ can be extended to a function in $C^+(\R)$.
Choose such an extension and also denote it by $f$.  For $n\in\N$ we set
\[
  g_n(t) \defeq f(t)|t-\la|^{1+1/n},\qquad t\in\R.
\]
Then each $g_n$ is continuous and positive.
Note that for $|t| \ge 2|\mu|$, where $\mu=\Re\la$, we have
$|t-\la| \ge |t-\mu| \ge |t|-|\mu| \ge \tfrac{1}{2}|t|$
%Note that for $|t|\ge 4|\Re\la|$ we have $|t-\la|^2\ge (t-\Re\la)^2\ge t^2/2$
and hence $|t-\la|f(t)\ge \tfrac{1}{2}|t|f(t)$,
which is bounded below by a positive constant.
This implies that $\lim_{|t|\to\infty}|g_n(t)| = \infty$.
Thus, for each $n\in\N$ there exists $t_n\in\sigma_V(T)$
such that $|g_n(t_n)| = \delta(g_n) \defeq \inf_{t\in\sigma_V(T)}|g_n(t)|$.

On the other hand, $\dom |T-\la|^{1+1/n}\subseteq\dom g_n(T)$
and $f(T) = g_n(T)|T-\la|^{-1-1/n}\in L(\calH_1)$.  For arbitrary $h\in\calH_2$ with $\|h\|=1$
define the positive measure $\mu_h \defeq \|E_T(\cdot)Vh\|^2$, which has support
contained in $\sigma_V(T)$. Then,
\begin{align*}
  \|f(T)Vh\|^2
  &= \|g_n(T)|T-\la|^{-1-1/n}Vh\|^2
  = \int_{\sigma_V(T)}\frac{|g_n(t)|^2}{|t-\la|^{2+2/n}}\,\rd\mu_h(t)\\
  &\ge\delta^2(g_n)\big\||T-\la|^{-1-1/n}Vh\big\|^2,
\end{align*}
and hence
\begin{align}
\begin{split}\label{e:useme}
  \|f(T)V\|
  &\ge \delta(g_n)\big\||T-\la|^{-1-1/n}V\big\| = |g_n(t_n)|\big\||T-\la|^{-1-1/n}V\big\| \\[0.5ex]
  &= |t_n-\la|f(t_n)\cdot |t_n-\la|^{1/n}\big\||T-\la|^{-1-1/n}V\big\|.
\end{split}
\end{align}
Now, consider the functions $h_n(t) \defeq |t-\la|^{-1-1/n}$, $n\in\N$,
and $h(t) \defeq |t-\la|^{-1}$.
Since $|t-\la|^{-1}\le|\Im\la|^{-1}$ for all $t\in\R$,
it follows from Lemma~\ref{l:helpy} that there exists $C>0$
such that $|h_n(t)-h(t)|\le C/n$ for all $t\in\R$.
This, together with (a), implies that
%Since the functions $u_n(x)=x^{\frac{n+1}{n}}$ converge uniformly to $u(x)=x$ in any closed interval of the form $[0,c]$, $c>0$, and
%$|t-\la|^{-1}\le|\Im\la|^{-1}$ for all $t\in\R$,  there exists $C>0$ such that $|h_n(t)-h(t)|\le C/n$ for all $t\in\R$. This implies
\begin{align*}
  1-\big\||T-\la|^{-1-1/n}V\big\|
  &\le \big\|(T-\la)^{-1}V\big\| - \big\||T-\la|^{-1-1/n}V\big\| \\[0.5ex]
  &= \|h(T)V\| - \|h_n(T)V\|
  \le \big\|h(T)V-h_n(T)V\big\|
  \le\frac{C\|V\|}{n}\,,
\end{align*}
which, for sufficiently large $n$, yields
\[
  \big\||T-\la|^{-1-1/n}V\big\|^n \ge \left(1 - \frac{C\|V\|}{n}\right)^n.
\]
As the right-hand side tends to $e^{-C\|V\|}$ as $n\to\infty$,
there exists $\gamma > 0$ such that $\big\||T-\la|^{-1-1/n}V\big\|\ge\gamma^{1/n}$
for all $n\in\N$.
Hence, if there exists some $n\in\N$ such that $|t_n-\la|\ge 1/\gamma$,
we find from \eqref{e:useme} that $\|f(T)V\|\,\ge\,|t_n-\la|f(t_n)$,
which means that $\la\in \mathcal{B}_{f(t_n)^{-1}\|f(T)V\|}(t_n)$.
Otherwise, there exists a subsequence $(t_{n_k})$ such that $t_{n_k}\to t_0$
as $k\to\infty$ with $t_0\in\sigma_V(T)$.
In this case, replacing $n$ by $n_k$ in \eqref{e:useme} and letting $k\to\infty$
we obtain
\[
  \|f(T)V\|\,\ge\,|t_0-\la|f(t_0)\cdot \|(T-\la)^{-1}V\|\,\ge\,|t_0-\la|f(t_0),
\]
that is, $\la\in \mathcal{B}_{f(t_0)^{-1}\|f(T)V\|}(t_0)$.
%is continuous, positive, bounded and $\|f(T)V\| =  \|(T-\la)^{-1}V\|$. However, $\lim_{|t|\to\infty}|tf(t)| = 1$ and so $f\notin C^+$. Therefore, we define the functions $f_n\in C^+$, $n\in\N$, by $f_n(t) \defeq |t-\la|^{-1+1/n}$, $t\in\R$. Also define $g_n(x) \defeq x^{1-1/n}$ and $g(x) = x$ for $x\in [0,c]$, where $c = |\Im\la|^{-1}$. If $c\le 1$, for $x\in (0,c]$ we have
%$$
%|g_n(x)-g(x)| = \left|e^{(1-\frac 1n)\log x}-e^{\log x}\right|\,\le\,\frac 1n\cdot x^{1-\frac 1n}|\log x|\,\le\,\frac 1n\cdot \sqrt x|\log x|\,\le\,\frac{2}{ne}.
%$$
%If $c>1$, then the same estimate holds for $x\le 1$. For $1 < x\le c$ we have
%$$
%|g_n(x)-g(x)| = \left|e^{(1-\frac 1n)\log x}-e^{\log x}\right|\,\le\,\frac 1nx|\log x|\,\le\,\frac{c|\log c|}{n}.
%$$
%Hence,
%$$
%|f_n(t)-f(t)| = |g_n(|t-\la|^{-1}) - g(|t-\la|^{-1})|\,\le\,\frac Cn
%$$
%for all $t\in\R$, where $C \defeq \max\{2/e,c|\log c|\}$. That is, $f_n\to f$ uniformly on $\R$. This implies $\|f_n(T)V\|\to\|f(T)V\|$ as $n\to\infty$.
%
%Towards a contradiction, suppose that $\|f(T)V\| = \|(T-\la)^{-1}V\| < 1$. Then there exist $N\in\N$ and $\veps > 0$ such that $\|f_n(T)V\|\le 1-\veps$ for $n\ge N$. By (b), for each $n\ge N$ there exists $t_n\in\sigma_V(T)$ such that
%$$
%|\la-t_n|\,\le\,f_n(t_n)^{-1}\|f_n(T)V\|\,\le\,|t_n-\la|^{1-\frac 1n}(1-\veps),
%$$
%that is, $|t_n-\la|^{1/n}\le 1-\veps$ for $n\ge N$. Consequently, $|t_n-\la|\le(1-\veps)^n\to 0$ as $n\to\infty$, which is the desired contradiction.
\end{proof}

Proposition~\ref{p:late} and Theorem~\ref{t:new} now immediately imply
the following slight improvement of Theorem~\ref{t:gersh_f}.

% -------------------------------------------------------------------
\begin{thm}\label{t:gersh_f2}
Let $S$ be the block operator matrix in \eqref{bom_gersh_1}.
Then, for any $f\in C^+(\sigma(A))$ and $g\in C^+(\sigma(D))$ we have
\begin{equation}\label{e:general2}
  \sigma(S)\setminus\R
  \subseteq \Biggl(\bigcup_{t\in\sigma_B(A)}\mathcal{B}_{f(t)^{-1}\|f(A)B\|}(t)\Biggr)\,\cap\,
  \Biggl(\bigcup_{s\in\sigma_{B^*}(D)}\mathcal{B}_{g(s)^{-1}\|g(D)B^*\|}(s)\Biggr).
\end{equation}
\end{thm}

We shall now check the performance of several spectral enclosures
for block operator matrices from above and from the literature on a specific example.
The result is illustrated in Figure~\ref{fig:ex_gershgorin} below.

% -------------------------------------------------------------------
\begin{ex}\label{ex:gershgorin}
Let $\cH_+=\cH_-=\dC^2$, consider the matrices
\[
  A = \mymatrix{1 & 0 \\ 0 & 2}, \quad
  B = \mymatrix{\frac{1}{3} & 0 \\[0.5ex] 0 & \frac{2}{3}}, \quad
  D = \mymatrix{1 & 0 \\ 0 & 1},
\]
and let $S$ be as in \eqref{bom_gersh_1}. The eigenvalues of $S$ are given by
\begin{equation}\label{ex_gersh_ev}
  1 \pm \frac{1}{3}i
  \qquad\text{and}\qquad
  \frac{3}{2} \pm \frac{\sqrt{7}}{6}i.
\end{equation}
\begin{enumerate}[1.]
\item % -----
The spectral enclosure from \cite[Theorem 2.7]{Salas99} states that
\begin{align*}
  \sigma(S)\setminus\R
  &\subseteq\{\la\in\rho(A) : \|(A-\la)^{-1}\|^{-1}\le\|B\|\}\,\cup\,
  \{\la\in\rho(D) : \|(D-\la)^{-1}\|^{-1}\le\|B\|\} \\[0.5ex]
  &= \mathcal{B}_{\frac 2 3}(1)\cup \mathcal{B}_{\frac 2 3}(2).
\end{align*}

\item % -----
The spectral enclosure from \cite[Theorem 6.4]{RT18} is slightly better than the previous one:
\begin{align*}
  \sigma(S)\setminus\R
  \subseteq\{\la\in\rho(A) : \|(A-\la)^{-1}B\|\ge 1\}\,\cup\,
  \{\la\in\rho(D) : \|(D-\la)^{-1}B^*\|\ge 1\}.
\end{align*}
However, since
\begin{equation}\label{e:bothies}
  \|(A-\la)^{-1}B\| = \frac 13\max\{|1-\la|^{-1},2|2-\la|^{-1}\}
  \qquad\text{and}\qquad
  \|(D-\la)^{-1}B^*\| = \frac{2}{3|1-\la|}\,,
\end{equation}
this yields the same enclosure as before: $\sigma(S)\setminus\R\subseteq\mathcal{B}_{\frac 2 3}(1)\cup \mathcal{B}_{\frac 2 3}(2)$.

\item % -----
The enclosure in \cite[Theorem~3.5]{BPT13} (see also \eqref{e:bpt})
yields the estimate
\[
  \sigma(S)\setminus\dR \subseteq (\mathcal{B}_{\frac{2}{3}}(1) \cup \mathcal{B}_{\frac{2}{3}}(2))\cap \mathcal{B}_{\frac 2 3}(1) = \mathcal{B}_{\frac 2 3}(1).
\]

\item % -----
The next enclosure that we check is \eqref{e:neumann}.
Since all matrices are diagonal, we have $N(\la) = N(\ol\la) = M(\la) = M(\ol\la)$.
Thus \eqref{e:neumann} is
\begin{equation}\label{encl_DMS15}
  \sigma(S)\setminus\R\,\subseteq\,\bigl\{\lambda\in\C\setminus\R : \|B(D-\lambda)^{-1}B^*(A-\lambda)^{-1}\| \ge 1\bigr\}.
\end{equation}
We have
\[
  B(D-\lambda)^{-1}B^*(A-\lambda)^{-1}
  = \frac{1}{9}\mymatrix{ (1-\lambda)^{-2} & 0 \\[2ex]
  0 & 4(1-\lambda)^{-1}(2-\lambda)^{-1} }
\]
and hence
\[
  \|B(D-\lambda)^{-1}B^*(A-\lambda)^{-1}\|
  = \frac 1 {9}\,\max\biggl\{\frac{1}{|\lambda-1|^2},\,\frac{4}{|\la-1|\,|\lambda-2|}\biggr\}.
\]
Therefore a non-real complex number $\lambda$ is in the right-hand side of \eqref{encl_DMS15}
if and only if
\[
  9|\lambda-1|^2 \le 1 \qquad\text{or}\qquad
  9|\lambda-1|\,|\lambda-2| \le 4.
\]
Since the first inequality implies the second, we obtain that \eqref{encl_DMS15} is equivalent to
\begin{equation}\label{ex_cassini}
  \sigma(S)\setminus\R \subseteq \biggl\{\lambda\in\dC\setminus\R :
  |\lambda-1|\,|\lambda-2| \le \frac{4}{9}\biggr\}.
\end{equation}

\item % -----
To compute \eqref{e:best} in Theorem~\ref{t:new}, we use \eqref{e:bothies} to get
\begin{equation}\label{e:debest}
  \sigma(S)\setminus\R\,\subseteq\,(\mathcal{B}_{\frac 1 3}(1)\cup \mathcal{B}_{\frac 2 3}(2))\cap \mathcal{B}_{\frac 23}(1)
  = \mathcal{B}_{\frac 1 3}(1)\cup(\mathcal{B}_{\frac 2 3}(1)\cap \mathcal{B}_{\frac 23}(2)).
\end{equation}

\item % -----
Let us now discuss our spectral enclosure from Theorem \ref{t:gersh_f}.
Choose $g(t) = f(t) = |t|^{-1}$, which is valid since $A$ and $D$ are invertible.  Then
\[
  \|f(A)B\| = \|A^{-1}B\| = \frac{1}{3}
  \qquad\text{and}\qquad
  \|g(D)B^*\| = \|D^{-1}B^*\|=\frac{2}{3}.
\]
Hence, \eqref{e:general} yields
\begin{equation}\label{Mozart}
  \sigma(S)\setminus\R
  \subseteq \bigl(\mathcal{B}_{\frac 1 3}(1)\cup \mathcal{B}_{\frac 2 3}(2)\bigr)\cap \mathcal{B}_{\frac 2 3}(1)
  = \mathcal{B}_{\frac 1 3}(1)\cup \bigl(\mathcal{B}_{\frac 2 3}(1)\cap \mathcal{B}_{\frac 2 3}(2)\bigr),
\end{equation}
which is the same as \eqref{e:debest}.

The right-hand side of \eqref{e:debest} (or \eqref{Mozart}) is obviously contained
in the right-hand side of \eqref{ex_cassini}; actually, it is significantly smaller
(e.g.\ the interval $(\tfrac 53,\tfrac 73)$ is in the right-hand side of \eqref{ex_cassini}
but not in the right-hand side of \eqref{Mozart}).
Note that the first four enclosing sets have the eigenvalues $1 \pm \frac{1}{3}i$
in their interior, while all four eigenvalues of $S$ lie on the boundary
of the region given in \eqref{e:debest}.
\end{enumerate}

% -------------------------------------------------------------------
\begin{figure}[ht]
\begin{center}
\begin{tikzpicture}[scale=2.5]
  \filldraw[color=enclcol, thick, fill=enclcol!10] (1,0) circle (0.333);
  \filldraw[color=enclcol, thick, fill=enclcol!10] (1.5,0.441) arc (138.6:221.4:0.667)
  arc (-41.41:41.41:0.667);
  \draw[enclcol, very thin] (2,0) circle (0.666);
  \draw[enclcol, very thin] (1,0) circle (0.666);
  \draw[color=blue, densely dashed, thick, variable=\t, domain=-90:270, smooth]
    plot ({1.5+cos(\t)*sqrt(cos(2*\t)/4+sqrt(0.1975-sin(2*\t)*sin(2*\t)/16)},
    {sin(\t)*sqrt(cos(2*\t)/4+sqrt(0.1975-sin(2*\t)*sin(2*\t)/16)});
  \draw[->] (-0.5,0) -- (3,0);
  \draw[->] (0,-0.8) -- (0,0.8);
  \filldraw[black] (1,0.333) circle (0.03);
  \filldraw[black] (1,-0.333) circle (0.03);
  \filldraw[black] (1.5,0.441) circle (0.03);
  \filldraw[black] (1.5,-0.441) circle (0.03);
  \draw[black] (1,-0.03) -- (1,0.03);
  \draw[black] (2,-0.03) -- (2,0.03);
  \node at (1,-0.12) {$1$};
  \node at (2,-0.12) {$2$};
\end{tikzpicture}
\caption{The spectral enclosure for $\sigma(S)\setminus\dR$ in \eqref{e:debest}
for the operator in Example~\ref{ex:gershgorin} is a union of a disc and the intersection
of two discs (filled orange region).
The boundary of the set on the right-hand side of \eqref{ex_cassini} is the blue dashed line.
The eigenvalues of $S$ (see \eqref{ex_gersh_ev}) are indicated with black dots.}
\label{fig:ex_gershgorin}
\end{center}
\end{figure}
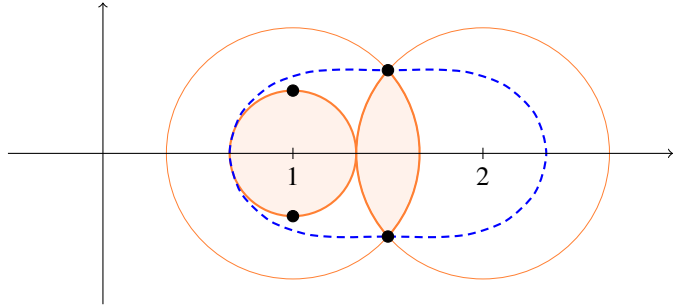
\end{ex}

In the following corollary we consider a useful special case of Theorem~\ref{t:gersh_f}.
We denote by $\C_+$ and $\C_-$ the open right and left half-planes, respectively.
The enclosure described in Corollary~\ref{thm:gershgorin} is illustrated in
Figure~\ref{fig:gershgorin17}.

% -------------------------------------------------------------------
\begin{cor}\label{thm:gershgorin}
Let $S$ be the block operator matrix in \eqref{bom_gersh_1} and assume that $0\in\rho(A)$.
Then
\begin{equation}\label{specincldiscs}
  \sigma(S)\setminus\dR \subseteq \bigcup_{a\in\sigma_B(A)} \mathcal{B}_{|a|\|A^{-1}B\|}(a).
\end{equation}
Assume, in addition, that $\|A^{-1}B\|<1$.  Then the right-hand side of \eqref{specincldiscs}
is contained in the set
\begin{equation}\label{double_sector}
  \Biggl\{z\in\dC: |\Im z| \le \frac{\|A^{-1}B\|}{\sqrt{1-\|A^{-1}B\|^2}\,}|\Re z|\Biggr\},
\end{equation}
which is a double-sector with half opening angle $\arcsin\|A^{-1}B\|$;
moreover,
\[
  \bigl(\sigma(S)\setminus\dR\bigr)\cap\mathbb C_+ \subseteq   \begin{cases}
    \bigl\{z\in\dC:\ \Re z\ge\bigl(1-\|A^{-1}B\|\bigr)\min\bigl(\sigma(A)\cap(0,\infty)\bigr)\bigr\},
    & \sigma(A)\cap(0,\infty)\ne\varnothing
    \\[1ex]
     \varnothing, & \text{otherwise},
  \end{cases}
\]
and
\[
  \bigl(\sigma(S)\setminus\dR\bigr)\cap\mathbb C_- \subseteq
  \begin{cases}
    \bigl\{z\in\dC:\ \Re z\le\bigl(1-\|A^{-1}B\|\bigr)\max\bigl(\sigma(A)\cap(-\infty,0)\bigr)\bigr\},
    & \sigma(A)\cap(-\infty,0)\ne\varnothing
    \\[1ex]
     \varnothing, & \text{otherwise}.
  \end{cases}
\]
%Finally, let $\mathbb C_+$ \braces{$\C_-$} stand for the open right \braces{left} half-plane.
%If $\sigma(A)\cap(0,\infty)\ne\varnothing$ then
%\begin{equation*}
%  \bigl(\sigma(S)\setminus\dR\bigr)\cap\mathbb C_+ \subseteq
%     \bigl\{z\in\dC:\ \Re z\ge\bigl(1-\|A^{-1}B\|\bigr)\min\bigl(\sigma(A)\cap(0,\infty)\bigr)\bigr\},
%\end{equation*}
%and if  $\sigma(A)\cap(-\infty,0)\ne\varnothing$ then
%\begin{equation*}
%  \bigl(\sigma(S)\setminus\dR\bigr)\cap\mathbb C_- \subseteq
%    \bigl\{z\in\dC:\ \Re z\le\bigl(1-\|A^{-1}B\|\bigr)\max\bigl(\sigma(A)\cap(-\infty,0)\bigr)\bigr\}.
%\end{equation*}
\end{cor}

\begin{proof}
Since $0\in\rho(A)$, the inclusion \eqref{specincldiscs} follows
from Theorem~\ref{t:gersh_f} by setting $f(t)=|t|^{-1}$.
Now assume also that $\|A^{-1}B\|<1$.
It is elementary to check that the lines
\[
  \Im z = \pm \frac{\|A^{-1}B\|}{\sqrt{1-\|A^{-1}B\|^2}\,}\Re z
\]
touch the discs $\mathcal{B}_{|a|\|A^{-1}B\|}(a)$, $a\in\sigma(A)$,  tangentially.
Further, these discs are contained in the double sector enclosed by
the two lines (see \eqref{double_sector}).
Hence, the right-hand side of \eqref{specincldiscs} is contained in \eqref{double_sector}.

Finally, if $\sigma(A)\cap(0,\infty)\ne\varnothing$, then,
for every $a\in\sigma(A)\cap(0,\infty)$ and $z\in \mathcal{B}_{a\|A^{-1}B\|}(a)$, we have
\[
  \Re z \ge a-a\|A^{-1}B\| \ge \bigl(1-\|A^{-1}B\|\bigr)\min\bigl(\sigma(A)\cap(0,\infty)\bigr).
\]
Similarly, if $\sigma(A)\cap(-\infty,0)\ne\varnothing$, then,
for every $a\in\sigma(A)\cap(-\infty,0)$ and $z\in \mathcal{B}_{a\|A^{-1}B\|}(a)$, we have
\[
  \Re z \le \bigl(1-\|A^{-1}B\|\bigr)\min\bigl(\sigma(A)\cap(-\infty,0)\bigr).
\]
This shows the inclusions for $(\sigma(S)\setminus\dR)\cap\mathbb C_+$
and $(\sigma(S)\setminus\dR)\cap\mathbb C_-$.
\end{proof}

A similar result holds when one replaces $A$ and $B$ by $D$ and $B^*$, respectively.
More precisely,  if $0\in\rho(D)$ then
\begin{equation}\label{specincldiscs2}
  \sigma(S)\setminus\dR \subseteq
  \bigcup_{d\in\sigma_{B^*}(D)} \mathcal{B}_{|d|\|D^{-1}B^*\|}(d).
\end{equation}
If it is also assumed that $\|D^{-1}B^*\|<1$, then
\begin{equation}\label{double_sector2}
  \sigma(S)\setminus\dR \subseteq \Biggl\{z\in\dC:\ |\Im z| \le \frac{\|D^{-1}B^*\|}{\sqrt{1-\|D^{-1}B^*\|^2}\,}|\Re z|\Biggr\},
\end{equation}
which is a double-sector with half opening angle $\arcsin\|D^{-1}B^*\|$;
moreover,
\[
  \bigl(\sigma(S)\setminus\dR\bigr)\cap\mathbb C_+ \subseteq   \begin{cases}
    \bigl\{z\in\dC:\ \Re z\ge\bigl(1-\|D^{-1}B^*\|\bigr)\min\bigl(\sigma(D)\cap(0,\infty)\bigr)\bigr\},
    & \sigma(D)\cap(0,\infty)\ne\varnothing
    \\[1ex]
     \varnothing, & \text{otherwise},
  \end{cases}
\]
and
\[
  \bigl(\sigma(S)\setminus\dR\bigr)\cap\mathbb C_- \subseteq
  \begin{cases}
    \bigl\{z\in\dC:\ \Re z\le\bigl(1-\|D^{-1}B^*\|\bigr)\max\bigl(\sigma(D)\cap(-\infty,0)\bigr)\bigr\},
    & \sigma(D)\cap(-\infty,0)\ne\varnothing
    \\[1ex]
     \varnothing, & \text{otherwise}.
  \end{cases}
\]
%
%Moreover, if $\sigma(D)\cap(0,\infty)\ne\varnothing$ then
%\begin{equation*}
  %\bigl(\sigma(S)\setminus\dR\bigr)\cap\mathbb C^+ \subseteq
     %\bigl\{z\in\dC:\ \Re z\ge\bigl(1-\|D^{-1}B^*\|\bigr)\min\bigl(\sigma(D)\cap(0,\infty)\bigr)\bigr\},
%\end{equation*}
%and if  $\sigma(D)\cap(-\infty,0)\ne\varnothing$ then
%\begin{equation*}
  %\bigl(\sigma(S)\setminus\dR\bigr)\cap\mathbb C^- \subseteq
    %\bigl\{z\in\dC:\ \Re z\le\bigl(1-\|D^{-1}B^*\|\bigr)\max\bigl(\sigma(D)\cap(-\infty,0)\bigr)\bigr\}.
%\end{equation*}
%
%

% *******************************************************************
\section{Application to \texorpdfstring{$J$}{J}-frame operators}
\label{sec:jframe}
% *******************************************************************

%Frames were introduced by Duffin and Schaeffer in
%\cite{DS52} and later Daubechies, Grossmann, and Meyer showed in \cite{DGM85} the deep
%importance for data processing, for details we refer to \cite{C03}.
%A \emph{frame} for a Hilbert space with inner product $\langle \cdot,\cdot\rangle$
%is a family of
%vectors $\cF=\{f_i\}_{i\in I}$  for which there exist
%constants $0 < \alpha \le \beta < \infty$ such that for every $f$ in
%the Hilbert space we have
%\begin{equation}\label{ecu frames}
%  \alpha\,\|f\|^2 \le \sum_{i\in I} |\langle f,f_i\rangle |^2 \le \beta\,\|f\|^2.
%\end{equation}
%Then usually one considers the two operators
%\[ Tx = \sum_{i\in I}x_if_i, \text{ for } x=(x_i)_i\in\ell_2(I)  \qquad \text{and }
%  S \defeq TT^*,\]
%where, by \eqref{ecu frames}, $T$ is bounded. $T$ is called the \emph{synthesis operator}
% and $S$ the \emph{frame operator}.

\noindent
Originally, frame theory has been developed for Hilbert spaces;
see, e.g.\ \cite{C03} and the references therein.
A \emph{frame} for a Hilbert space $(\cH,\Skdef)$
is a family of
vectors $\cF=\{f_i\}_{i\in I}$  for which there exist
constants $0 < \alpha \le \beta < \infty$ such that
\begin{equation}\label{ecu frames}
  \alpha\,\|f\|^2 \le \sum_{i\in I} |\langle f,f_i\rangle |^2 \le \beta\,\|f\|^2, \quad \text{for every $f\in\cH$}.
\end{equation}
The optimal constants $\alpha$ and $\beta$ for which \eqref{ecu frames} holds are known as the \emph{frame bounds} of $\cF$.

Recently, various approaches have been suggested to introduce frame theory
also to Krein spaces; see \cite{EFW,GMMM12,PW}.
In this section we apply our results to $J$-frame operators as introduced
in \cite{GMMM12}; see also \cite{GLLMMT}.
%In particular, we improve a result from \cite{GLLMMT} and present a
%better spectral enclosure for the $J$-frame operator;
%see Theorem \ref{Lemmens} and Theorem \ref{Lemmens2} below.
In particular, we improve the enclosure for the non-real spectrum of $J$-frame operators
obtained in \cite{GLLMMT}; see Theorem \ref{Lemmens} below.

\medskip%In general, the operator $S$ in \eqref{PlazaMalvinas} is not self-adjoint with respect to the standard Hilbert space
%inner product $\Skdef$ in $\calH_+ \oplus \calH_-$, but it is self-adjoint with respect to the indefinite inner product induced by the block operator matrix
%\[
  %J=\matriz{I}{0}{0}{-I}.
%\]
%In other words, if we introduce the indefinite inner product
%\begin{equation}\label{Saavedra}
  %[x, y] \defeq (Jx, y), \qquad x,y \in \calH,
%\end{equation}
%the space $\calH$ turns into a Krein space and
%the operator $S$ is self-adjoint in the Krein space $(\calH,\Skindef)$;
%see, e.g.\ \cite{AI89,B}.
%Conversely, every bounded self-adjoint operator in a Krein space
%can be written in the form \eqref{PlazaMalvinas} with respect to
%a fundamental decomposition of the Krein space.
%
%\bigskip

An indefinite inner product space $(\cH,\product)$ is a (complex) vector space $\cH$ endowed with a Hermitian sesquilinear form $\Skindef$.
%A vector $x\in\cH$ is positive, negative or neutral according to the value of $\K{x}{x}$.
Given a subspace $\cS$ of $\cH$, the orthogonal subspace to $\cS$ is defined by
\[
\cS^{\ort}=\{x\in\cH:\ \K{x}{s}=0 \ \text{for every $s\in\cS$} \},
\]
and $\cS$ is called non-degenerate if $\cS\cap\cS^{\ort}=\{0\}$.
If $\cS$ and $\cT$ are subspaces of $\cH$, the notation $\cS \ort \cT$ stands for $\cS\subseteq \cT^{\ort}$.

A Krein space is a non-degenerate indefinite inner product space $(\cH,\product)$ which admits a decomposition $\calH = \calH_+\,\ds\,\calH_-$ such that $\calH_+\,[\perp]\calH_-$
%$\calH_+$ is orthogonal to $\calH_-$ with respect to $\Skindef$
and $(\calH_\pm,\pm\product)$ are Hilbert spaces.
Such a decomposition is
often called a \emph{fundamental decomposition} and it is denoted $\calH = \calH_+\,[\ds]\,\calH_-$.

The Hilbert spaces $(\calH_\pm,\pm\product)$ induce in a natural way a positive definite inner product $\Skdef$ on $\cH$ such that $(\cH,\Skdef)$ is a Hilbert space. Observe that the inner products $\Skindef$ and $\Skdef$ of $\cH$ are related by means of a
\emph{fundamental symmetry}, i.e.\ a unitary self-adjoint operator $J\in L(\mathcal H)$
that satisfies
\[
  (f,g)=\K{Jf}{g}, \quad f,g\in\cH.
\]
Although the fundamental decomposition is not unique, the norms induced
by different fundamental decompositions turn out to be equivalent;
see, e.g.\ \cite[Proposition~I.1.2]{L82}.
Therefore, the (Hilbert space) topology in $\cH$ does not depend on the
chosen fundamental decomposition.

%If $P_\pm$ denotes the projection onto $\calH_\pm$ with respect to a fundamental decomposition,
%then $J = P_+ - P_-$ is called a \emph{fundamental symmetry} for $\calH$.
%
%Given a fundamental decomposition $\cH=\cH_+ [\dotplus] \cH_-$ of $\cH$,
%let $\Skdef$ be the (definite) inner-product defined by
%\[
  %(f,g)=\K{Jf}{g}, \quad f,g\in\cH.
%\]
\medskip

Let us now introduce $J$-frames. Given a Krein space $(\cH,\Skindef)$, consider a frame $\cF=\{f_i\}_{i\in I}$ for the associated Hilbert space $(\cH,\Skdef)$
and set
\[
  I_+ \defeq \{i\in I:\, [f_i,f_i] \ge 0\}
  \qquad\text{and}\qquad
  I_- \defeq \{i\in I:\, [f_i,f_i]< 0\}.
\]
Then $\cF$ is called a \emph{$J$-frame} for $\cH$
if $\cM_+\defeq \ol{\linspan\{f_i:\, i\in I_+\}}$
and $\cM_-\defeq \ol{\linspan\{f_i:\, i\in I_-\}}$  are non-degenerate subspaces
of $\cH$ and there exist constants $0<\alpha_\pm \leq \beta_\pm$ such that
\begin{equation}\label{eq J frame bounds}
  \alpha_\pm (\pm[f,f]) \le \sum_{i\in I_\pm} \big|[f,f_i]\big|^2
  \le \beta_\pm (\pm[f,f])\,
  \qquad \text{for $f\in \cM_\pm$};
\end{equation}
see \cite[Theorem~3.9]{GMMM12}.
The spaces $(\calM_\pm,\pm\product)$ are then Hilbert spaces
by \cite[Proposition~3.8]{GMMM12} and the optimal constants $0<\alpha_\pm \leq \beta_\pm$
are called the \emph{$J$-frame bounds of $\cF$}.

Note that \eqref{eq J frame bounds} says that $\cF_+=\{f_i\}_{i\in I_+}$
and $\cF_-=\{f_i\}_{i\in I_-}$ are frames for the Hilbert spaces $(\cM_+,\Skindef)$
and $(\cM_-,-\Skindef)$, respectively.
Moreover, the frame bounds for $\cF_+$ and $\cF_-$ are $\alpha_+,\beta_+$
and $\alpha_-,\beta_-$, respectively. Also note that not necessarily $\calM_+\,[\perp]\,\calM_-$.

The \emph{$J$-frame operator} associated with $\cF$ is defined by
\begin{equation*}
  Sf =\sum_{i\in I_+} [f,f_i]f_i- \sum_{i\in I_-} [f,f_i]f_i, \qquad f\in\cH.
\end{equation*}
It plays a fundamental role in the indefinite reconstruction formula (see \cite{GMMM12}).
The operator $S$ is an invertible, bounded, self-adjoint operator in the Krein space $\cH$.
The following representation for $J$-frame operators was obtained
in \cite[Theorems~3.1 and 3.2]{GLLMMT}.

% -------------------------------------------------------------------
\begin{thm}\label{Dublin}
Given a bounded self-adjoint operator $S$ in a Krein space
$(\mathcal H, \Skindef)$, the following conditions are equivalent.
\begin{myenum}
\item
$S$ is a $J$-frame operator.
\item
There exists a fundamental decomposition
\begin{equation}\label{decomp17}
  \mathcal H = \mathcal H_+ [\Pluspunkt] \mathcal H_-
\end{equation}
such that $S$ admits a representation with respect to \eqref{decomp17} of the form
\begin{equation}\label{Cork}
  S = \mymatrix{ A & -AK \\[0.5ex] K^*A & D}
\end{equation}
where $A$ is a uniformly positive operator in the Hilbert space
$(\mathcal H_+, \Skindef)$, $K: \mathcal H_-\to \mathcal H_+$ is a
uniform contraction\footnote{The operator norm used depends on the norm induced by the
respective fundamental decomposition \eqref{decomp17} or \eqref{decomp18}.\label{fn:norm}}
\textup{(}i.e.\ $\|K\|<1$\textup{)}, and $D$ is a self-adjoint operator
such that $D+K^*AK$ is uniformly positive in the
Hilbert space $(\mathcal H_-, -\Skindef)$.
\item
There exists a fundamental decomposition
\begin{equation}\label{decomp18}
  \mathcal H = \mathcal K_+ [\Pluspunkt] \mathcal K_-
\end{equation}
such that $S$ admits a representation with respect to \eqref{decomp18}
of the form
\begin{equation}\label{Corky}
  S = \mymatrix{ A' & LD' \\[0.5ex] -D'L^* & D' }
\end{equation}
where $D'$ is a uniformly positive operator in $(\mathcal K_-, -\Skindef)$,
$L: \mathcal K_-\to \mathcal K_+$ is a uniform contraction\textsuperscript{\rmref{fn:norm}},
and $A'$ is a self-adjoint operator such that $A'+LD'L^*$ is uniformly positive
in $(\mathcal K_+, \Skindef)$.
\end{myenum}
\end{thm}

\medskip

The representations for the $J$-frame operator given in Theorem~\ref{Dublin}
were used to show that the $J$-frame bounds for $\cF$ are related to the boundary
of the spectrum of the uniformly positive operators $D+K^*AK$ and $A'+LD'L^*$.
More precisely, \cite[Proposition~4.1]{GLLMMT} says that if $S$ is
represented as in \eqref{Cork}, then
\begin{equation}\label{cotasZ}
  \alpha_- = \min\,\sigma(D + K^*AK)
  \qquad\text{and}\qquad
  \beta_- = \max\,\sigma(D + K^*AK).
\end{equation}
On the other hand, if $S$ is represented as in \eqref{Corky}, then
\begin{equation}\label{cotasZ'}
  \alpha_+ = \min\,\sigma(A'+LD'L^*)
  \qquad\text{and}\qquad
  \beta_+ = \max\,\sigma(A'+LD'L^*).
\end{equation}

Given a $J$-frame $\cF=\{f_i\}_{i\in I}$ for $\cH$ %for a Krein space $\cH$
with $J$-frame operator $S$, the \emph{canonical dual $J$-frame} of $\cF$
is defined as $\cF'=\{S^{-1}f_i\}_{i\in I}$.
It is also a $J$-frame for $\cH$ such that $\cF'_{\pm}=\{S^{-1}f_i\}_{i\in I_{\pm}}$
are frames for $(\mathcal{M}_{\mp}^{\ort},\pm \Skindef)$,
i.e.\ there exist constants $0<\gamma_\pm\le\delta_\pm$ such that
\[
  \gamma_\pm (\pm[f,f]) \leq \sum_{i\in I_\pm} \big|[f,S^{-1}f_i]\big|^2
  \leq \delta_\pm (\pm[f,f]) \qquad \text{for every $f\in \cM_\mp^{\ort}$}.
\]
%The optimal constants $0<\gamma_\pm \leq \delta_\pm$ are the $J$-frame bounds of $\cF'$.
%They are also related to the representations in Theorem~\ref{Dublin}.
%In fact, \cite[Proposition~4.2]{GLLMMT} states that if $S$ is represented
%as in \eqref{Corky}, then
The $J$-frame bounds of $\cF'$ are also related to the representations in Theorem~\ref{Dublin}:
if $S$ is represented as in \eqref{Corky} then
\begin{equation}\label{cotasD}
  \gamma_- = \min\,\sigma\bigl((D')^{-1}\bigr) = \bigl(\max\,\sigma(D')\bigr)^{-1}
  \quad\text{and}\quad
  \delta_- = \max\,\sigma\bigl((D')^{-1}\bigr) = \bigl(\min\,\sigma(D')\bigr)^{-1},
\end{equation}
and if $S$ is represented as in \eqref{Cork} then
\begin{equation}\label{cotasA}
  \gamma_+ = \min\,\sigma(A^{-1}) = \bigl(\max\,\sigma(A)\bigr)^{-1}
  \quad\text{and}\quad
  \delta_+ = \max\,\sigma(A^{-1}) = \bigl(\min\,\sigma(A)\bigr)^{-1};
\end{equation}
see \cite[Proposition~4.2]{GLLMMT}.

The following theorem gives an enclosure for the non-real spectrum of
the $J$-frame operator $S$ of a $J$-frame $\cF$ in terms of the $J$-frame bounds
associated with $\cF$ and its canonical dual $J$-frame $\cF'$.
%It is worth noticing that the opening angle of the sector in which the
%enclosure is contained depends on the norm of the angular operators $K$ and $L$.

% -------------------------------------------------------------------
\begin{thm}\label{Lemmens}
Let $\cF$ be a $J$-frame for $(\cH,\Skindef)$ with $J$-frame operator $S$. Then,
\begin{equation}\label{Saveedra12}
  \sigma(S)\setminus\R\,\subseteq\,\left(\bigcup_{a\in [\delta_+^{-1},\gamma_+^{-1}]}
  \mathcal{B}_{a\|K\|}(a)\right)
  \,\cap\,
  \left(\bigcup_{b\in [\alpha_-,\beta_-]} \mathcal{B}_{\tfrac{b\|K\|}{1-\|K\|^2}}\left(\tfrac{b}{1-\|K\|^2}\right)\right),
\end{equation}
where $K$ is the angular operator appearing in \eqref{Cork}.
Also,
\begin{equation}\label{Saveedra13}
  \sigma(S)\setminus\R\,\subseteq\,\left(\bigcup_{d\in [\delta_-^{-1},\gamma_-^{-1}]} \mathcal{B}_{d\|L\|}(d)\right)
  \,\cap\,
  \left(\bigcup_{b\in [\alpha_+,\beta_+]} \mathcal{B}_{\tfrac{b\|L\|}{1-\|L\|^2}}\left(\tfrac{b}{1-\|L\|^2}\right)\right),
\end{equation}
where $L$ is the angular operator appearing in \eqref{Corky}.
The sets on the right-hand sides of \eqref{Saveedra12} and \eqref{Saveedra13} are contained
in sectors of the form $\{z\in\C_+:|\Im z| \le \tan\varphi\cdot \Re z\}$
with half opening angles $\varphi=\arcsin\|K\|$ and $\varphi=\arcsin\|L\|$, respectively.
\end{thm}

\begin{proof}
Let $\sigma_B(A)$ be defined as before Proposition \ref{p:late}. Obviously
we have $\sigma_B(A)\subseteq \sigma(A)\subseteq [a_-,a_+]$, where the constants
$a_\pm$ are defined in \eqref{defconstapm}.
Applying Corollary \ref{thm:gershgorin} to $S$ represented as in \eqref{Cork} we obtain that
\begin{equation}\label{JF1}
  \sigma(S)\setminus\R\,\subseteq\,\bigcup_{a\in [a_-,a_+]} \mathcal{B}_{a\|K\|}(a).
\end{equation}
Moreover, according to \eqref{cotasA} we have that $a_-=\delta_+^{-1}$ and $a_+=\gamma_+^{-1}$.
On the other hand, $S^{-1}$ is also a $J$-frame operator and it is easy to check that
\[
  S^{-1} = \mymatrix{
    A^{-1} - KZK^* & KZ \\[0.5ex]
    -ZK^* & Z },
\]
where $Z \defeq (D + K^*AK)^{-1}$ is a uniformly positive operator; cf.\ Theorem \ref{Dublin}.
Therefore Corollary~\ref{thm:gershgorin} applied to $S^{-1}$ represented as above implies
\[
  \sigma(S^{-1})\setminus\R\,\subseteq\,
  \bigcup_{r\in [r_-,r_+]}\mathcal{B}_{r\|K^*\|}\left(r\right),
\]
where $[r_-,r_+]$ is the closure of the numerical range of $Z$.  Also, \eqref{cotasZ} says that $r_-=\beta_-^{-1}$
and $r_+=\alpha_-^{-1}$.  With $b\defeq\tfrac{1}{r}$ it follows that
\[
  \sigma(S^{-1})\setminus\R\,\subseteq\,
  %\bigcup_{\tfrac{1}{b}\in [\beta_-^{-1},\alpha_-^{-1}]} B_{\tfrac{\|K\|}{b}}\left(\tfrac{1}{b}\right) =
  \bigcup_{b\in [\alpha_-,\beta_-]} \mathcal{B}_{\tfrac{\|K\|}{b}}\left(\tfrac{1}{b}\right).
\]
Recall that $\la\in\sigma(S)\setminus\{0\}$ if and only
if $\tfrac{1}{\la}\in\sigma(S^{-1})\setminus\{0\}$. Moreover,
observe that for $r>0$ %we have
\[
  \frac{1}{\lambda} \in \mathcal{B}_{\tfrac{\|K\|}{r}}\left(\tfrac{1}{r}\right)
  \quad \text{if and only if} \quad \lambda \in
  \mathcal{B}_{\tfrac{r\|K\|}{1-\|K\|^2}}\left(\tfrac{r}{1-\|K\|^2}\right).
\]
Therefore,
\begin{equation}\label{JF2}
  \sigma(S)\setminus\R\,\subseteq\,
  \bigcup_{b\in [\alpha_-,\beta_-]} \mathcal{B}_{\tfrac{b\|K\|}{1-\|K\|^2}}\left(\tfrac{b}{1-\|K\|^2}\right),
\end{equation}
and \eqref{Saveedra12} follows by intersecting \eqref{JF1} and \eqref{JF2}.

The proof of \eqref{Saveedra13} is similar.  It follows from Corollary~\ref{thm:gershgorin}
applied to $S$ represented as in \eqref{Corky}, and also to $S^{-1}$ represented as
\[
  S^{-1} = \mymatrix{
    Z'  & -Z'L \\[0.5ex]
    L^* Z' & (D')^{-1}-L^*Z'L },
\]
with $Z' = (A'+LD'L^*)^{-1}$. % and similar arguments used above.

The statement about the sectors is clear from Corollary~\ref{thm:gershgorin}.
\end{proof}

In the following, we compare Theorem \ref{Lemmens} with the enclosure for
the non-real spectrum of $J$-frame operators obtained in \cite{GLLMMT}.

Let $\cF$ be a $J$-frame for a Krein space $(\cH,\Skindef)$
with $J$-frame operator $S$ and $J$-frame bounds $0<\alpha_\pm\leq\beta_\pm$.
Assume also that $0<\gamma_\pm\leq\delta_\pm$ are the $J$-frame bounds of its
canonical dual $J$-frame $\cF'$. %Representing $S$ as in \eqref{Cork},
In \cite[Corollary~5.3]{GLLMMT} it was shown that
\begin{equation}\label{real2}
  \sigma(S)\setminus\R\,\subseteq\, \mathring{
  \mathcal{B}}_{\min \{\gamma_+^{-1}, \gamma_-^{-1}\}}\bigl(\min \{\gamma_+^{-1}, \gamma_-^{-1}\}\bigr)
  \cap\left\{\la \in\C : \Re \la \ge\frac{\max\{\alpha_+,\alpha_-\}}{2}\right\}.
\end{equation}
Here, $\mathring{\mathcal{B}}_r(a)$ denotes the interior of $\mathcal{B}_r(a)$.
Let us show that the intersection of the sets on the right-hand sides
of \eqref{Saveedra12} and \eqref{Saveedra13}
are (strictly) contained in the right-hand side of \eqref{real2}.

%First, recall that by \eqref{Saveedra12}, $\sigma(S)\setminus \R$ is contained in
%\[
%\left(\bigcup_{a\in [\delta_+^{-1},\gamma_+^{-1}]}
  %B_{a\|K\|}(a)\right)
  %\,\cap\,
  %\left(\bigcup_{b\in [\alpha_-,\beta_-]} B_{\tfrac{b\|K\|}{1-\|K\|^2}}\left(\tfrac{b}{1-\|K\|^2}\right)\right).
%\]
\medskip
For every $a\in [\delta_+^{-1},\gamma_+^{-1}]$ it is easy to see
that $\mathcal{B}_{a\|K\|}(a)$ is strictly contained in $\mathring{\mathcal{B}}_{\gamma_+^{-1}}(\gamma_+^{-1})$.
Therefore,
\[
  \bigcup_{a\in [\delta_+^{-1},\gamma_+^{-1}]}\mathcal{B}_{a\|K\|}(a)\subseteq
  \mathring{\mathcal{B}}_{\gamma_+^{-1}}(\gamma_+^{-1}).
\]
On the other hand, given $r>0$, if $\la\in \mathcal{B}_{\tfrac{r\|K\|}{1-\|K\|^2}}\left(\tfrac{r}{1-\|K\|^2}\right)$,
then $\Re\la \geq \tfrac{r}{1+\|K\|}>\tfrac{r}{2}$.  Thus,
\[
  \bigcup_{b\in [\alpha_-,\beta_-]} \mathcal{B}_{\tfrac{b\|K\|}{1-\|K\|^2}}\left(\tfrac{b}{1-\|K\|^2}\right)
  \subseteq \left\{\la \in\C : \Re \la \ge\frac{\alpha_-}{2}\right\}.
\]
Similarly, it is easy to see that
\[
  \bigcup_{d\in [\delta_-^{-1},\gamma_-^{-1}]} \mathcal{B}_{d\|L\|}(d) \subseteq
  \mathring{\mathcal{B}}_{\gamma_-^{-1}}(\gamma_-^{-1}),
\]
and
\[
  \left(\bigcup_{b\in [\alpha_+,\beta_+]} \mathcal{B}_{\tfrac{b\|L\|}{1-\|L\|^2}}\left(\tfrac{b}{1-\|L\|^2}\right)\right)
  \subseteq \left\{\la \in\C : \Re \la \ge\frac{\alpha_+}{2}\right\}.
\]
Hence, Theorem~\ref{Lemmens} improves the enclosure \eqref{real2} for the non-real spectrum
of the $J$-frame operator $S$ obtained in \cite{GLLMMT}.

\section*{Acknowledgements}

\noindent F.\ Mart\'{\i}nez Per\'{\i}a and C.\ Trunk  gratefully acknowledge the support of the DFG
(Deutsche Forschungsgemeinschaft)  from the project TR 903/21-1.
 In addition, J.\ I.\ Giribet and F.\ Mart\'{\i}nez Per\'{\i}a gratefully acknowledges the support from the grant  PIP CONICET 0168.
%F. Philipp was supported by MinCyT Argentina under grant PICT-2014-1480.

% *******************************************************************

\end{document}